\documentclass{amsart}
\usepackage[a4paper]{geometry}
\usepackage{tikz}
\usetikzlibrary{arrows, calc, positioning, decorations.markings}
\tikzset{
    between/.style args={#1 and #2}{
         at = ($(#1)!0.5!(#2)$)
    }
}

\usepackage{ucs}
\usepackage[utf8x]{inputenc}
\usepackage{amsthm}
\usepackage{url}
\usepackage{comment}
\usepackage{enumitem}
\makeatletter

\def\part{\@startsection{part}{0}%
  \z@{\linespacing\@plus\linespacing}{.5\linespacing}%
  {\normalfont\bfseries\centering}}
\makeatother
\newtheorem{thm}{Theorem}[section]
\newtheorem{lem}[thm]{Lemma}
\newtheorem{lemma}[thm]{Lemma}
\newtheorem{prop}[thm]{Proposition}
\newtheorem{cor}[thm]{Corollary}
\newtheorem{conj}{Conjecture}
\newtheorem*{asp}{Aspiration}
\theoremstyle{remark}
\newtheorem{remark}{Remark}

\theoremstyle{definition}

\newcommand{\cy}{\mathcal{Y}}

\newcommand{\cm}{\mathcal{M}}
\newcommand{\bt}{\mathbb{T}}

\newcommand{\be}{\mathbb{E}}
\newcommand{\norm}[1]{\left\| #1\right\|}
\newcommand{\abs}[1]{\left| #1\right|}
\newcommand{\dmu}{\mathrm{d}\mu}

\usepackage{hyperref}
\usepackage{xcolor}
\hypersetup{
  colorlinks   = true, 
  urlcolor     = blue, 
  linkcolor    = blue, 
  citecolor   = red 
}

\usepackage{setspace}
\usepackage{amsmath}
\usepackage{amsfonts}
\usepackage{amssymb}

\DeclareMathOperator{\conv}{conv}

\newtheorem{definition}[thm]{{\bf {\em Definition}}}


\title[Disentanglement, Multilinear Duality and Factorisation]{Disentanglement,  Multilinear Duality and Factorisation
  for non-positive operators}
\author{Anthony Carbery, Timo S. H\"anninen and Stef\'an Ingi Valdimarsson}
\address{Anthony Carbery, 
School of Mathematics and Maxwell Institute for Mathematical Sciences, 
University of Edinburgh,
James Clerk Maxwell Building, 
Peter Guthrie Tait Road,
King's Buildings, 
Mayfield Road, 
Edinburgh, EH9 3FD, 
Scotland.}
\email{A.Carbery@ed.ac.uk}
\address{Timo S. H\"anninen,
Department of Mathematics and Statistics, University of Helsinki, P.O. Box 68, FI-00014 Helsinki, Finland; and School of Mathematics and Maxwell Institute for Mathematical Sciences, 
University of Edinburgh,
James Clerk Maxwell Building, 
Peter Guthrie Tait Road,
King's Buildings, 
Mayfield Road, 
Edinburgh, EH9 3FD, 
Scotland.}
\email{timo.s.hanninen@helsinki.fi}
\address{Stef\'an Ingi Valdimarsson,
Arion banki, Borgart\'un 19, 105 Reykjav\'ik,
Iceland; and Science
Institute, University of Iceland, Dunhagi 5, 107
Reykjav\'ik, Iceland.}

\email{sivaldimarsson@gmail.com}
\begin{document}
\setcounter{tocdepth}{1}
\begin{abstract}
In our previous work \cite{CHV1} we established a multilinear duality and factorisation theory for norm inequalities for pointwise weighted geometric means of positive linear operators defined on normed lattices. In this paper we extend the reach of the theory for the first time to the setting of {\em general} linear operators defined on normed spaces. The scope of this theory includes multilinear Fourier restriction-type inequalities. We also sharpen our previous theory of positive operators.

\medskip
\noindent
Our results all share a common theme: estimates on a weighted geometric mean of linear operators can be {\em disentangled}
into quantitatively linked estimates on each operator separately. The concept of disentanglement recurs throughout  the paper.

\medskip
\noindent
The methods we used in the previous work -- principally convex optimisation -- relied strongly on positivity. In contrast, in this paper we use a vector-valued reformulation of disentanglement, geometric properties (Rademacher-type) of the underlying normed spaces, and probabilistic considerations related to $p$-stable random variables.

\end{abstract}
\maketitle
\setcounter{tocdepth}{1}
\tableofcontents
\section{Introduction}\label{intro}

In our previous work \cite{CHV1} we introduced and developed a general functional-analytic principle concerning norm inequalities for pointwise
weighted geometric means
$$ \prod_{j=1}^d |T_j f_j(x)|^{\alpha_j}$$
of {\em positive} linear operators $T_j$ defined on suitable spaces, where $\alpha_j \geq 0$
and $\sum_{j=1}^d \alpha_j = 1$. In this paper we extend our study to the situation in which the linear operators $T_j$ are no longer assumed to be positive. The techniques of \cite{CHV1} relied strongly on positivity, so it will be necessary to involve a new set of ideas.

\medskip
\noindent
In order to set the scene for this, it will be helpful to recall the main theorem of \cite{CHV1}, but we first we need to set up some notation.
Let $(X, {\rm d} \mu)$ be a measure space and let $\mathcal{M}(X)$ be the class of measurable functions on $X$. Let $\mathcal{Y}$ be a real or complex normed space.
(For example, if $Y$ is a measure space, $\mathcal{Y}$ could be the class $\mathcal{S}(Y)$ of simple
functions with an $L^p$-norm for some $p \geq 1$.) We say that a  linear map
$T: \mathcal{Y} \to \mathcal{M}(X)$ {\bf saturates $X$} if, for each subset
  $E \subseteq X$ of positive measure, there exists a subset 
  $E^\prime  \subseteq E$ with $\mu(E^\prime) > 0$ and an $h \in \mathcal{Y}$
  such that $|Th| > 0$ a.e. on $E^\prime$. For reasons explained in \cite{CHV1}, such a condition is needed for the result which follows to hold.

\begin{thm}(\cite{CHV1})
\label{thmmainbaby}
Suppose that $X$ is a $\sigma$-finite measure space and that $\mathcal{Y}_j$, for $j=1,\dots,d$, are normed lattices. Suppose that the linear
operators $T_j: \mathcal{Y}_j \to \mathcal{M}(X)$ are positive
and that each $T_j$ saturates $X$. Suppose that $ 0 < q \leq \infty$ and 
$\sum_{j=1}^d \alpha_j = 1$. Finally, suppose that
\begin{equation}\label{abc}
\left\|\prod_{j=1}^d (T_jf_j)^{\alpha_j}\right\|_{L^q(X)}
\leq A \prod_{j=1}^d \Big\|f_j\Big\|_{\mathcal{Y}_j}^{\alpha_j}
\end{equation}
for all nonnegative $f_j \in \mathcal{Y}_j$, $1 \leq j \leq d$.

\medskip
\noindent
\begin{enumerate}
\item[{\bf Case I.}] {\bf (Disentanglement).} If $q = 1$, then there exist nonnegative measurable functions 
$g_j$ on $X$ such that 
\begin{equation}\label{factorisebaby_0}
1 \leq \prod_{j=1}^d g_j(x)^{\alpha_j}\qquad\mbox{a.e. on $X$,}
\end{equation}
and such that for each $j$, 
\begin{equation}\label{controlbaby_0}
\int_X g_j(x)T_jf_j(x) {\rm d}\mu(x) \leq A\|f_j\|_{\mathcal{Y}_j}
\end{equation}
for all $f_j \in  \mathcal{Y}_j$, with the same constant $A$ as in \eqref{abc}.

\medskip
\noindent
Conversely, if the $T_j$ are positive linear operators such that there exist nonnegative measurable functions 
$g_j$ on $X$ such that \eqref{factorisebaby_0} holds, and such that \eqref{controlbaby_0} holds
for all $f_j \in \mathcal{Y}_j$, then \eqref{abc} holds for all nonnegative $f_j \in \mathcal{Y}_j$.

\bigskip
\item[{\bf Case II.}] {\bf (Multilinear Duality).} If $q > 1$, then for every nonnegative
 $G \in L^{q^\prime}(X)$ there exist nonnegative measurable functions 
$g_j$ on $X$ such that 
\begin{equation}\label{factorisebaby}
G(x) \leq \prod_{j=1}^d g_j(x)^{\alpha_j}\qquad\mbox{a.e. on $X$,}
\end{equation}
and such that for each $j$, 
\begin{equation}\label{controlbaby}
\int_X g_j(x)T_jf_j(x) {\rm d}\mu(x) \leq A\|G\|_{L^{q^{\prime}}}\|f_j\|_{\mathcal{Y}_j}
\end{equation}
for all $f_j \in  \mathcal{Y}_j$, with the same constant $A$ as in \eqref{abc}.

\medskip
\noindent
Conversely, if the $T_j$ are positive linear operators such that for every nonnegative
 $G \in L^{q^\prime}(X)$ there exist nonnegative measurable functions 
$g_j$ on $X$ such that \eqref{factorisebaby} holds, and such that \eqref{controlbaby} holds
for all $f_j \in \mathcal{Y}_j$, then \eqref{abc} holds for all nonnegative $f_j \in \mathcal{Y}_j$. 

\bigskip
\item[{\bf Case III.}] {\bf (Multilinear Maurey Factorisation).} If $0 < q < 1$, then there exist nonnegative measurable
  functions $g_j$ on $X$ such that\footnote{We  caution that we use the notation $\|g \|_q:=(\int|g|^q)^{1/q}$ and $q' :=q/(q-1)$ for $q<0$ and for $0<q<1$, even though in these cases $\| \cdot \|_q$  does not define a norm.} 
  \begin{equation}\label{factorisebaby_b}
    \| \prod_{j=1}^d g_j(x)^{\alpha_j} \|_{q'} = 1
    \end{equation}
  and such that for each $j$, \eqref{controlbaby_0} holds for all $f_j \in \mathcal{Y}_j$,  with the same constant $A$ as in \eqref{abc}.

  \medskip
  \noindent
  Conversely, if the $T_j$ are positive linear operators such that there exist nonnegative measurable functions 
$g_j$ on $X$ such that \eqref{factorisebaby_b} holds, and such that \eqref{controlbaby_0} holds
for all $f_j \in \mathcal{Y}_j$, then \eqref{abc} holds for all nonnegative $f_j \in \mathcal{Y}_j$. 

\end{enumerate}
\end{thm}
Numerous illustrations and applications of this theorem were given in \cite{CHV1}. It should be stressed that this result is a general one, applying to the class of positive operators broadly.

\medskip
\noindent
The forward parts of this result are the difficult ones; the converses
follow easily by applying H\"older's inequality. When $d=1$, Case II reduces to
an elementary duality statement concerning the operator $T: \mathcal{Y}\to L^q$ and this gives rise to the sobriquet
``multilinear duality'' in the case of general $d$. The term ``factorisation'' relates both to the pointwise factorisation expressed by
\eqref{factorisebaby} and to the condition \eqref{controlbaby} which is a statement that each operator $T_j$
factorises through a certain weighted $L^1$-space.

\medskip
\noindent
Case I, corresponding to $q=1$, plays a special role, and indeed the remaining cases corresponding to $q \neq 1$
can be deduced from it without too much difficulty -- see Section~\ref{revisit} for arguments of this type. 
We describe the case $q=1$ as a ``disentanglement'' result since it disentangles a bound \eqref{abc} on the pointwise combination of the $T_j$'s into bounds \eqref{controlbaby_0} on each $T_j$ separately, with the individual bounds linked via \eqref{factorisebaby_0}.

\medskip
\noindent
Notice that, when suitably modified, the statement of Theorem~\ref{thmmainbaby} makes
perfectly good sense in principle {\em without} the hypothesis of positivity of the
operators $T_j$; nevertheless, as we have mentioned, the arguments from \cite{CHV1} rely very
heavily on positivity. In this paper we use vector-valued techniques to
develop an analogue of Theorem~\ref{thmmainbaby} which applies to general linear operators defined on normed spaces. See Theorems~\ref{answer:general_operators}, \ref{lin_q}, \ref{qwm_ba} and \ref{thm:lq_linear} below.

\medskip
\noindent
In what follows we shall primarily focus on the case of $L^1$ norms of pointwise weighted products  $\prod_{j=1}^d|T_jf_j|^{\gamma_j}$ in our pursuit of extending Theorem~\ref{thmmainbaby} to general linear operators $T_j$. We return to the case of general $L^q$-norms of such expressions in Section~\ref{revisit}, and there we see that it is relatively straightforward to derive the results for general $q$, which even in the positive case significantly generalise Theorem~\ref{thmmainbaby}, from those corresponding to $q=1$.

\medskip
\noindent
We next give a simple lemma. All of our main results can be framed as reversals of the implication it establishes (under various auxiliary hypotheses).

\medskip
\noindent
\begin{lem}\label{prelim}
  Let $\mathcal{Y}_j$ be normed spaces and let $T_j : \mathcal{Y}_j \to \mathcal{M}(X)$ be linear mappings for $1 \leq j \leq d$. Suppose $\gamma_j > 0$ are given. 
  Assume that for some $(p_j)$ with  $0 < p_j < \infty $ we have the condition 
  \begin{equation}\label{homog}
    \sum_{j=1}^d \frac{\gamma_j}{p_j} = 1,
\end{equation}
  and that there exist nonnegative measurable functions $(\phi_j)$ on $X$ such that
\begin{equation}\label{yan1}
\prod_{j=1}^d \phi_j(x)^{\gamma_j/p_j} \geq 1
  \end{equation}
a.e. on $X$ and such that
\begin{equation}\label{yan2}
  \left(\int_X |T_j f_j(x)|^{p_j} \phi_j(x) {\rm d} \mu(x) \right)^{1/p_j} \leq A \|f_j\|_{\mathcal{Y}_j}
 \end{equation}
for all $f_j \in \mathcal{Y}_j$. Then
\begin{equation}\label{yan}
\int_X \prod_{j=1}^d |T_j f_j(x)|^{\gamma_j} {\rm d} \mu(x) \leq A^{\sum_{j=1}^d \gamma_j}\prod_{j=1}^d \|f_j \|_{\mathcal{Y}_{j}}^{\gamma_j}
  \end{equation}
for all $f_j \in \mathcal{Y}_j$.
  \end{lem}

\begin{proof}
  Let $\theta_j = \gamma_j/p_j$. Then $\sum_{j=1}^d \theta_j = 1$, and, by \eqref{yan1}, \eqref{yan2} and H\"older's inequality, we have
  $$\int_X \prod_{j=1}^d |T_j f_j(x)|^{\gamma_j} {\rm d} \mu(x)
\leq \int_X \prod_{j=1}^d |T_j f_j(x)|^{\gamma_j}  \phi_j(x)^{\gamma_j/p_j}{\rm d} \mu(x)$$
$$= \int_X \prod_{j=1}^d |T_j f_j(x)|^{p_j \theta_j} \phi_j(x)^{\theta_j} {\rm d} \mu(x)
\leq  \prod_{j=1}^d   \left(\int_X |T_j f_j(x)|^{p_j} \phi_j(x){\rm d} \mu(x)\right)^{\theta_j}$$
$$ \leq A^{\sum_{j=1}^d p_j \theta_j} \prod_{j=1}^d \|f_j\|_{\mathcal{Y}_j}^{p_j \theta_j} = A^{\sum_{j=1}^d \gamma_j} \prod_{j=1}^d \|f_j\|_{\mathcal{Y}_j}^{\gamma_j }.$$
\end{proof}
Taking
$\gamma_j = q \alpha_j$ with $q$ and $\sum_{j=1}^d\alpha_j = 1$ as in the preceding discussion makes a point of contact with Theorem~\ref{thmmainbaby}.

\medskip
\noindent
Note that Lemma~\ref{prelim} has no content in the linear case $d=1$. Our main concern will therefore be with the converse scenario in the genuinely multilinear case $d\geq 2$. The lemma delineates what we might hope for. More precisely:

\medskip
\noindent
{\bf {Basic Question.}}
Let $d \geq 2$. Suppose $X$ is a $\sigma$-finite measure space, $\mathcal{Y}_j$ are normed spaces, $T_j : \mathcal{Y}_j \to \mathcal{M}(X)$ are saturating linear mappings, and $\gamma_j > 0$ for $1 \leq j \leq d$. We suppose that \eqref{yan} holds. For which $(p_j)$ (if any) with $0 < p_j < \infty $ satisfying condition \eqref{homog} 
can we conclude that there exist nonnegative $(\phi_j)$ such that conditions \eqref{yan1} and \eqref{yan2} hold, perhaps with a loss in the constants? 

\medskip
\noindent
Once again we emphasise that we ask this question in the broad context: we seek answers which do not rely upon the precise nature of the operators $T_j:\cy_j\to \cm(X)$, but instead which will hold universally over a wide class of linear operators. We expect that the set of admissible exponents $(p_j)$, in addition to satisfying \eqref{homog}\footnote{For a discussion of why we require this condition, see Proposition~\ref{homog_nec} in the Appendix.}, will reflect whatever geometric structures the normed spaces $\cy_j$ may possess. 

\medskip
\noindent
We shall give separate answers to this question in the settings of general linear operators and of positive linear operators. It transpires that in order to develop the theory for general linear operators,
it first makes sense to consider a related question for positive linear 
operators: if in Theorem~\ref{thmmainbaby} we take the lattices $\mathcal{Y}_j$ to be $L^{r_j}$-spaces, are there stronger, $r_j$-dependent, conclusions that we can make?

\medskip
\noindent
The following result answers our Basic Question for positive linear operators on Lebesgue spaces,   
with no loss in constants. A corresponding answer in the case of general linear operators on Lebesgue spaces is given in Theorem~\ref{answer:general_operators}.

\begin{thm}
\label{answer:qwm_0} 
Suppose that $X$ and $Y_j$, for $j=1,\dots,d$, are measure spaces and that $X$ is $\sigma$-finite.
Suppose that the linear
operators $T_j: \mathcal{S}(Y_j) \to \mathcal{M}(X)$ are positive
and that each $T_j$ saturates $X$. Suppose that $ 1 \leq r_j \leq \infty$ for all $j$. 
Finally, suppose that for some exponents $\gamma_j > 0$ we have
\begin{equation}\label{abc_17}
\int_X \prod_{j=1}^d (T_jf_j)(x)^{\gamma_j} {\rm d} \mu(x)
\leq A^{\sum_{j=1}^d \gamma_j} \prod_{j=1}^d \Big\|f_j\Big\|_{L^{r_j}(Y_j)}^{\gamma_j}
\end{equation}
for all nonnegative simple functions $f_j$ on ${Y}_j$, $1 \leq j \leq d$.

\medskip
\noindent
Then for all $(p_j)$ satisfying $ 0 <p_j < \infty$ for all $j$, $\sum_{j=1}^d \gamma_j /p_j = 1$ and $p_j \leq r_j$ for all $j$,
there exist nonnegative $(\phi_j)$ such that
\begin{equation}\label{yan12}
\prod_{j=1}^d \phi_j(x)^{\gamma_j/p_j} \geq 1
  \end{equation}
a.e. on $X$ and such that
\begin{equation}\label{yan22}
  \left(\int_X |T_j f_j(x)|^{p_j} \phi_j(x) {\rm d} \mu(x) \right)^{1/p_j} \leq A \|f_j\|_{r_j}
 \end{equation}
for all $f_j \in \mathcal{S}(Y_j)$.
\end{thm}

\medskip
\noindent
\begin{remark} In the Appendix below we give an example of positive linear operators $(T_j)$ satisfying \eqref{abc_17}, for which the set of $(p_j)$ satisfying $0 < p_j < \infty$ and $\sum_{j=1}^d \gamma_j /p_j = 1$, and for which the conclusion of Theorem~\ref{answer:qwm_0} holds, consists {\em precisely} of those satisfying $ p_j \leq r_j$ for every $j$. See Corollary~\ref{seis_tres}.  
Thus the condition $p_j \leq r_j$ is sharp if we want our result to hold broadly for positive operators without further reference to their individual properties.\footnote{For {\em particular} positive operators $(T_j)$, the result may hold even when $p_j > r_j$ for some $j$. Indeed, let $X = Y_j = [0,1]$ with Lebesgue measure, let $r_j = 1$ for all $j$ and let each $T_j$ be given by $T_jf = \int_0^1 f$, so that each $T_jf$ is constant on $[0,1]$. Then \eqref{abc_17} holds for all exponents $\gamma_j > 0$, with $A=1$. If we take $\phi_j(x) = 1$ for all $j$, then both \eqref{yan12} and \eqref{yan22} hold for all exponents $0<p_j<\infty$.} \end{remark}

\medskip
\noindent
Notice that the set
$$\left\{(p_j) \in (0,\infty)^d \, : \sum_{j=1}^d \frac{\gamma_j}{p_j} = 1 \mbox{ and } p_j \leq r_j \mbox{ for all }j  \right\}$$
is nonempty if and only if $\sum_{j=1}^d \gamma_j /r_j \leq 1$. In particular, Theorem~\ref{answer:qwm_0} has no content unless $\sum_{j=1}^d\gamma_j/r_j \leq 1$. In Corollary~\ref{seis_tres} we demonstrate, by example, that if  $\sum_{j=1}^d\gamma_j/r_j >1$, then the set of $(p_j)$ satisfying the conclusion of Theorem~\ref{answer:qwm_0} may indeed be empty.

\medskip
\noindent
Under hypothesis \eqref{abc_17}, the disentangled conclusions \eqref{yan22} for $p_j \leq \max \{r_j, \gamma_j\}$ alone, with otherwise unspecified but nontrivial $(\phi_j)$, are more straightforward, and can be established by methods which are not genuinely multilinear.\footnote{The range $p_j \leq \max \{r_j, \gamma_j\}$ for this simpler problem is also known to be sharp, as the arguments in the Appendix confirm.} 
The significant feature of Theorem~\ref{answer:qwm_0} is that under the hypotheses $\sum_{j=1}^d \gamma_j /p_j = 1$ and $p_j \leq r_j$ for all $j$, we can choose $(\phi_j)$ also satisfying the specific quantitative lower bound \eqref{yan12}. Similar remarks apply to our subsequent results.

\medskip
\noindent
We point out that the case $p_j = 1$ for all $j$ of Theorem ~\ref{answer:qwm_0} directly implies 
Case I (and therefore Case II) of Theorem~\ref{thmmainbaby} (in the special case where the spaces $\mathcal{Y}_j$ are taken to be $L^{r_j}$).
The case $p_j = r_j$ of Theorem~\ref{answer:qwm_0} is, however, the crucial one, and in a slightly different notation can be presented as follows:

\begin{thm}[Disentanglement for positive operators on Lebesgue spaces]\label{qwm_1} 
Suppose that $X$ and $Y_j$, for $j=1,\dots,d$, are measure spaces and that $X$ is $\sigma$-finite.
Suppose that the linear
operators $T_j: \mathcal{S}(Y_j) \to \mathcal{M}(X)$ are positive
and that each $T_j$ saturates $X$. Suppose that $1 \leq p_j < \infty$ for all $j$, and that $\theta_j \geq 0$ are such that $\sum_{j=1}^d \theta_j = 1$. Finally, suppose that
\begin{equation*}
\int_X \prod_{j=1}^d (T_jf_j)(x)^{p_j \theta_j} {\rm d} \mu(x)
\leq B \prod_{j=1}^d \Big\|f_j\Big\|_{L^{p_j}(Y_j)}^{p_j \theta_j}
\end{equation*}
for all nonnegative simple functions $f_j$ on ${Y}_j$, $1 \leq j \leq d$.
Then there exist nonnegative measurable functions 
$\phi_j$ on $X$ such that
\begin{equation*}
  \prod_{j=1}^d \phi_j(x)^{\theta_j} \geq 1
\end{equation*}
almost everywhere on $X$ and such that for each $j$, 
\begin{equation*}
\left(\int_X |T_jf_j(x)|^{p_j} \phi_j(x) {\rm d}\mu(x)\right)^{1/p_j} \leq B^{1/p_j} \|f_j\|_{L^{p_j}(Y_j)}
\end{equation*}
for all simple functions $f_j$ on ${Y}_j$.
\end{thm}

\medskip 
In analogy with the Case I of Theorem~\ref{thmmainbaby}, we shall also call this result a disentanglement theorem, and it is an instance of the general disentanglement theorem for positive operators on $p_j$-convex spaces which we shall present as Theorem~\ref{qwm_bb}.

\medskip
\noindent
As the reader will have noticed, by homogeneity we may take $B = 1$ (and $A=1$ in earlier results)
without loss. (And by playing with homogeneities the constant $B^{1/p_j}$ can be replaced with $B^{(\sum_{j=1}^d p_j \theta_j)^{-1}}$).

\medskip
\noindent
In order to address our main concern in the paper -- the extension of the theory to include general linear operators which are not necessarily positive -- we shall consider the analogous situation under hypotheses of Rademacher-type in place of $p$-convexity. Our use of $p$-convexity and Rademacher-type proceeds in parallel with their deployment in the development of the Maurey theory, see \cite{GCRdeF,kalton2016}. For now we state a sample theorem, 
which, in the case that the normed spaces $\mathcal{Y}_j$ are $L^{r_j}$-spaces, answers the  Basic Question. We shall significantly generalise this result later, see Theorem~\ref{qwm_ba}.

\begin{thm}
\label{answer:general_operators}
Suppose that $X$ and $Y_j$, for $j=1,\dots,d$, are measure spaces and that $X$ is $\sigma$-finite.
Suppose that $T_j: \mathcal{S}(Y_j) \to \mathcal{M}(X)$ are linear (not necessarily positive) operators
    and that each $T_j$ saturates $X$. Suppose that\footnote{The proof will reveal that the result remains valid under the weaker assumption $0 < r_j < \infty$, provided that we accordingly modify \eqref{condition:general_operators} to $0 < p_j < r_j$ for those $j$ for which $0< r_j < 2$.} $1 \leq r_j < \infty$ for all $j$. 
Finally, suppose that for some exponents $\gamma_j > 0$ we have
\begin{equation}\label{abc_17_a}
\int_X \prod_{j=1}^d |T_jf_j(x)|^{\gamma_j} {\rm d} \mu(x)
\leq A^{\sum_{j=1}^d \gamma_j} \prod_{j=1}^d \Big\|f_j\Big\|_{L^{r_j}(Y_j)}^{\gamma_j}
\end{equation}
for all simple functions $f_j$ on ${Y}_j$, $1 \leq j \leq d$.

\medskip
\noindent
Then for all $(p_j)$ such that  $\sum_{j=1}^d \gamma_j /p_j = 1$ and 
\begin{equation}
 \begin{cases} 
      0< p_j <r_j & \text{for those $j$ for which $1\leq r_j<2$} \\
      0< p_j \leq 2 & \text{for those $j$ for which $2\leq r_j<\infty$},\\ 
    \end{cases}
       \label{condition:general_operators}
\end{equation}

 there exist nonnegative $\phi_j$ such that
\begin{equation*}
\prod_{j=1}^d \phi_j(x)^{\gamma_j/p_j} \geq 1
  \end{equation*}
a.e. on $X$ and such that
\begin{equation*}
  \left(\int_X |T_j f_j(x)|^{p_j} \phi_j(x) {\rm d} \mu(x) \right)^{1/p_j}  \lesssim_{\{\gamma_j,r_j,p_j\}} A \|f_j\|_{L
   ^{r_j}(Y_j)}
 \end{equation*}
for all $f_j \in \mathcal{S}(Y_j)$.
\end{thm}
\medskip
\noindent

\noindent
\medskip

\begin{remark}
In the Appendix below we give an example of linear operators $(T_j)$ satisfying \eqref{abc_17_a}, for which the set of $(p_j)$ satisfying $0 < p_j < \infty$ and $\sum_{j=1}^d \gamma_j /p_j = 1$, and for which the conclusion of Theorem~\ref{answer:general_operators} holds, consists {\em precisely} of those satisfying \eqref{condition:general_operators}. See Corollary~\ref{seis_diez}. Thus the condition \eqref{condition:general_operators} is sharp if we want our result to hold broadly for linear operators without further reference to their individual properties. For specific operators $T_j$ the conclusion may nevertheless hold even if \eqref{condition:general_operators} is violated.
\end{remark}

\medskip

Note that the set of $(p_j)$ satisfying $\sum_{j=1}^d \gamma_j /p_j = 1$ together with \eqref{condition:general_operators} will be nonempty if and only if
\begin{equation*}
 \begin{cases} 
   \sum_{j=1}^d \gamma_j/\min\{r_j, 2\} \; < 1   & \text{when at least one $r_j<2$} \\
    \sum_{j=1}^d \gamma_j \leq 2  & \text{when all $2\leq r_j<\infty$}.\\  \end{cases}
      \end{equation*}
In Corollary~\ref{seis_diez} we demonstrate, by example, that if this condition is violated, the set of $(p_j)$ satisfying the conclusion of Theorem~\ref{answer:general_operators} may indeed be empty.

\medskip
\noindent
The special case of this result corresponding to $p_j =2$ for all $j$ is singled out:

\begin{thm}[Disentanglement for general linear operators on Lebesgue spaces]\label{lindis_2}
 Suppose that $X$ and $Y_j$, for $j=1,\dots,d$, are measure spaces and that $X$ is $\sigma$-finite.
Suppose that the linear
operators $T_j: \mathcal{S}(Y_j) \to \mathcal{M}(X)$ saturate $X$. Suppose that $\theta_j > 0$ and $\sum_{j=1}^d \theta_j = 1$. Finally, suppose that
for some exponents $2 \leq r_j < \infty$ we have
\begin{equation*}
\int_X \prod_{j=1}^d |T_jf_j(x)|^{2 \theta_j} {\rm d} \mu(x)
\leq B \prod_{j=1}^d \Big\|f_j\Big\|_{L^{r_j}(Y_j)}^{2\theta_j}
\end{equation*}
for all simple functions $f_j$ on ${Y}_j$, $1 \leq j \leq d$.
Then there exist nonnegative measurable functions 
$\phi_j$ on $X$ such that
\begin{equation*}
  \prod_{j=1}^d \phi_j(x)^{\theta_j} \geq 1
\end{equation*}
almost everywhere on $X$ and such that for each $j$, 
\begin{equation*}
\left(\int_X |T_jf_j(x)|^{2} \phi_j(x) {\rm d}\mu(x)\right)^{1/2} \lesssim B^{1/2} \|f_j\|_{L^{r_j}(Y_j)}
\end{equation*}
for all simple functions $f_j$ on ${Y}_j$.
  \end{thm}

Theorem~\ref{lindis_2} readily upgrades to the following result (see Section~\ref{revisit}), whose formulation can be compared to Case II of Theorem~\ref{thmmainbaby}:
\begin{thm}[Multilinear duality for general operators on Lebesgue spaces]\label{lin_q}
 Suppose that $X$ and $Y_j$, for $j=1,\dots,d$, are measure spaces and that $X$ is $\sigma$-finite.
Suppose that the linear
operators $T_j: \mathcal{S}(Y_j) \to \mathcal{M}(X)$ saturate $X$. Suppose that $\alpha_j > 0$ and  $\sum_{j=1}^d \alpha_j = 1$. Finally, suppose that
for some exponents $q \geq 2$ and $2 \leq r_j < \infty$ we have
\begin{equation*}
\left\|\prod_{j=1}^d |T_jf_j|^{\alpha_j}\right\|_q 
\leq B \prod_{j=1}^d \Big\|f_j\Big\|_{L^{r_j}(Y_j)}^{\alpha_j}
\end{equation*}
for all simple functions $f_j$ on ${Y}_j$, $1 \leq j \leq d$.
Then for every nonnegative $G \in L^{(q/2)'}$ there exist nonnegative measurable functions 
$g_j$ on $X$ such that
\begin{equation*}
  \prod_{j=1}^d g_j(x)^{\alpha_j} \geq G(x)
\end{equation*}
almost everywhere on $X$ and such that for each $j$, 
\begin{equation*}
\left(\int_X |T_jf_j(x)|^{2} g_j(x) {\rm d}\mu(x)\right)^{1/2} \lesssim B \|G\|_{(q/2)'}\|f_j\|_{L^{r_j}(Y_j)}
\end{equation*}
for all simple functions $f_j$ on ${Y}_j$.
  \end{thm}

The converse statements to these three results are once again also true, and are easy to verify. 

\medskip
\noindent
Note that in these last three results we do not assert ``~$\leq$~'' but only ``~$\lesssim$~'' in
the conclusions, and moreover the case $r_j = \infty$ is excluded from Theorems~\ref{answer:general_operators} and \ref{lin_q}. This is ultimately because we shall need to apply Khintchine's inequality.
Note also the numerology familiar from harmonic analysis, in which
$L^p$-boundedness of a positive operator for $p > 1$ (such as a maximal
operator) often corresponds to $L^{2p'}$ boundedness of a corresponding  
nonpositive operator (such as a singular integral operator). Even in the
linear case $d=1$, the duality statement is along the lines that
$T: L^r \to L^q$ with $q,r \geq 2$ if and only if
$\| |T^\ast g|^2\|_{q'/2} \lesssim \||g|^2\|_{r'/2}$ (rather than $\| T^\ast g\|_{q'} \lesssim \|g\|_{r'}$).

\subsection{Multilinear restriction and the Mizohata--Takeuchi conjecture}
As an indication of the scope of Theorem~\ref{lin_q}, we consider the so-called multilinear restriction problem for the Fourier transform. For $1 \leq j \leq n$, let $\Gamma_j: U_j \to \mathbb{R}^n$ (with $U_j \subseteq \mathbb{R}^{n-1}$) be smooth parametrisations of compact hypersurfaces $S_j$ in $\mathbb{R}^n$ with nonvanishing gaussian curvature. We assume that the hypersurfaces are transversal in
the sense that if $\omega_j(x)$ denotes a unit normal to $S_j$ at $x \in S_j$, then $|\omega_1(x_1) \wedge \dots \wedge \omega_n(x_n)| \geq c >0$
for all $x_j \in S_j$. The Fourier extension (or dual restriction) operator $\mathcal{E}_j$ for $S_j$ is given by
$$ \mathcal{E}_j f_j(x) = \int_{U_j} e^{2 \pi i x \cdot \Gamma(t_j)} f_j(t_j) {\rm d} t_j,$$

It is conjectured (see \cite{BCT}) that these operators satisfy the multilinear bound
\begin{equation}\label{multextn_0}
  \int_{\mathbb{R}^n} \prod_{j=1}^n |\mathcal{E}_j f_j(x)|^{2/(n-1)} {\rm d} x
  \lesssim   \prod_{j=1}^n \|f_j\|_{L^2(U_j)}^{2/(n-1)}
\end{equation}
or equivalently
\begin{equation}\label{multextn}
  \|\prod_{j=1}^n |\mathcal{E}_j f_j(x)|^{1/n}\|_{2n/(n-1)}
  \lesssim   \prod_{j=1}^n \|f_j\|_{L^2(U_j)}^{1/n}.
  \end{equation}
This is known up to endpoints (see \cite{BCT}, \cite{tao2019}) but is as yet unresolved in the form stated here.

\medskip
\noindent
These considerations clearly fit into the framework which we were discussing above, in particular
Theorem~\ref{lin_q}, and we therefore have the following:

\begin{thm}[Factorisation for multilinear restriction]\label{CFMR}
 The multilinear restriction bound \eqref{multextn} holds if and only if for all nonnegative $G \in L^n(\mathbb{R}^n)$, there exist nonnegative $g_1, \dots , g_n$ such that
  $$ \prod_{j=1}^n g_j(x)^{1/n} \geq G(x)$$
  a.e. and, for all $j$,
  $$ \left(\int_{\mathbb{R}^n} |\mathcal{E}_j f_j(x)|^{2} g_j(x) {\rm d} x \right)^{1/2}\lesssim \|G\|_n \|f_j \|_2.$$
  \end{thm}

On the other hand, the corresponding endpoint multilinear Kakeya theorem is due to Guth (\cite{Guth}, see also \cite{CV}). He proved it by directly establishing the following fundamental factorisation result:

\begin{thm}[Guth, \cite{Guth}\label{Guththm}]
For $1 \leq j \leq n$, let $\mathcal{T}_j$ be families of doubly-infinite tubes of unit cross-section with transversal directions. For all nonnegative $G \in L^n(\mathbb{R}^n)$, there exist nonnegative $g_1, \dots , g_n$ such that
  $$ \prod_{j=1}^n g_j(x)^{1/n} \geq G(x)$$
  a.e. and, for all $j$ and $T \in \mathcal{T}_j$, 
  $$ \int_{T} g_j(x) {\rm d} x \lesssim \|G\|_n.$$
\end{thm}

Moreover, coming from entirely different considerations, there is a conjecture, often attributed to Mizohata
and Takeuchi, which states:

\begin{conj}[Mizohata--Takeuchi conjecture]
  Let $S$ be a compact hypersurface of nonvanishing gaussian curvature, with corresponding Fourier extension operator $\mathcal{E}$. Then, for any nonnegative weight $w$ we have
  $$\int_{\mathbb{R}^n} |\mathcal{E} f(x)|^{2} w(x) {\rm d} x \lesssim \sup_T w(T) \int |f(t)|^2 {\rm d}t$$
  where the {\rm sup} is taken over all doubly-infinite tubes of unit cross-section with direction normal to $S$.
\end{conj}
Combining these last two statements we obtain:
\begin{prop}
Conditional on the Mizohata--Takeuchi conjecture, the multilinear restriction bound \eqref{multextn_0} holds. 
\end{prop}

  \begin{proof}
In order to establish \eqref{multextn_0}, we integrate the function $\prod_{j=1}^n |\mathcal{E}_j f_j(x)|^{2/n}$
against a test function $G$ in the unit ball of $L^n$.
We let $\mathcal{T}_j$ consist of tubes with directions normal to $S_j$. We apply Guth's theorem to $G$ obtain $g_j$ as in Theorem~\ref{Guththm}. Then
$$ \int_{\mathbb{R}^n} \prod_{j=1}^n |\mathcal{E}_j f_j(x)|^{2/n}G(x) {\rm d}x 
\leq  \int_{\mathbb{R}^n} \prod_{j=1}^n |\mathcal{E}_j f_j(x)|^{2/n}g_j(x)^{1/n} {\rm d}x \leq \prod_{j=1}^n  \left(\int_{\mathbb{R}^n} |\mathcal{E}_j f_j(x)|^2 g_j(x) {\rm d}x\right)^{1/n}$$
by H\"older's inequality. For each $j$ we have
$$\int_{\mathbb{R}^n} |\mathcal{E}_j f_j(x)|^2 g_j(x) {\rm d}x \lesssim \left(\sup_{T \in \mathcal{T}_j} \int_{T} g_j\right)  \int |f_j(t)|^2 {\rm d}t \lesssim \|f_j\|_2^2$$
by the Mizohata--Takeuchi conjecture and the second conclusion of Theorem
~\ref{Guththm}. Combining these estimates yields \eqref{multextn_0}.
  \end{proof}

\medskip
\noindent
\subsection{Structure of the paper}
In Section~\ref{sec:vector_valued_disentanglement} we first state and prove two results,  Theorem~\ref{functions} and Theorem~\ref{vector}, both equivalent to Case I of Theorem~\ref{thmmainbaby}, and then we indicate how we shall use vector-valued techniques to obtain our main theorems. In Section~\ref{sec:positive_operators_convexity} we discuss refinements of Theorem~\ref{thmmainbaby} for positive operators to the case of $p$-convex lattices; the main result here is Theorem~\ref{qwm_bb}. The case of general linear operators is taken up in Section~\ref{sec:general_operators_type}, and here we impose conditions of Rademacher-type; the main result in this setting is Theorem~\ref{qwm_ba}. In Section~\ref{revisit} we establish sharp multilinear duality and Maurey-type factorisation theorems for both positive and general linear operators, in Theorems~\ref{thm:lq_positive} and \ref{thm:lq_linear} respectively. The logical connections between these main results are summarised in Figure \ref{fig:M12}.

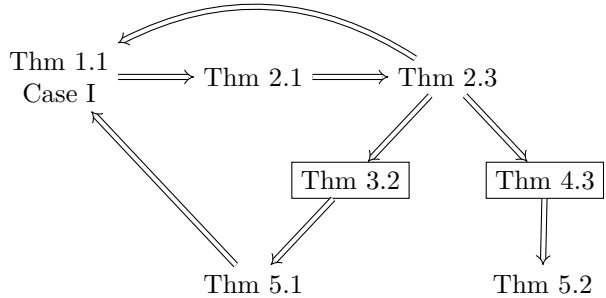
\begin{figure}[h]
\centering
\begin{tikzpicture}
  \node[align=center] (thm1p1) {Thm 1.1\\Case I};
  \node[right=of thm1p1] (thm2p1) {Thm 2.1};
  \node[right=of thm2p1] (thm2p2) {Thm 2.3};
  \node[below=of thm2p1] (dummy1) {};
  \node[below=of dummy1] (thm5p1) {Thm 5.1};
  \node[between=thm2p2 and thm5p1, rectangle, draw] (thm3p3) {Thm 3.2};
  \node[right=of thm3p3, rectangle, draw] (thm4p3) {Thm 4.3};
    \node[right=of thm5p1] (dummy2) {};
  \node[right=of dummy2] (thm5p2) {Thm 5.2};

  \draw[-implies,double equal sign distance] (thm1p1) to (thm2p1);
 \draw[-implies,double equal sign distance] (thm2p1) to (thm2p2);
  \draw[-implies,double equal sign distance, bend right] (thm2p2) to (thm1p1);
  \draw[-implies,double equal sign distance] (thm2p2) to (thm3p3);
  \draw[-implies,double equal sign distance] (thm3p3) to (thm5p1);
  \draw[-implies,double equal sign distance] (thm5p1) to (thm1p1);
  \draw[-implies,double equal sign distance] (thm2p2) to (thm4p3);
  \draw[-implies,double equal sign distance] (thm4p3) to (thm5p2);
\end{tikzpicture}
\caption{Taxonomy of main theorems
}
\label{fig:M12}
\end{figure}

The implications between the main result for positive operators on $p$-convex lattices, Theorem ~\ref{qwm_bb}, and its more basic manifestations Theorems~\ref{answer:qwm_0} and \ref{qwm_1} for $L^r$-spaces, are given in Figure \ref{fig:M11}.

\begin{figure}[h]
\centering
\begin{tikzpicture}
  \node (thm1p3) {Thm 1.3};
  \node[right=of thm1p3] (thm1p4) {Thm 1.4};
  \node[above=of thm1p4, rectangle, draw] (thm3p2) {Thm 3.2};
  \node[right=of thm3p2] (thm3p3) {};
  \node[right=of thm1p4] (thm5p1) {Thm 5.1};

  \draw[-implies,double equal sign distance] (thm1p3) to (thm1p4);
  \draw[-implies,double equal sign distance] (thm3p2) to (thm1p3);
 
  \draw[-implies,double equal sign distance] (thm3p2) to (thm5p1);
  \draw[-implies, double equal sign distance] (thm3p2) to (thm1p4);
\end{tikzpicture}
\caption{Positive operators
}
\label{fig:M11}
\end{figure}
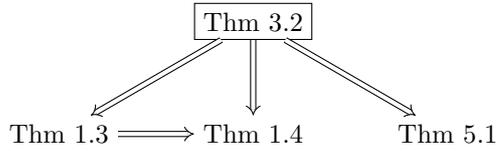

For general linear operators on normed spaces of (non-trivial) Rademacher-type, the corresponding logical implications between the main result, Theorem~\ref{qwm_ba} and the more basic manifestations Theorems~\ref{answer:general_operators}, \ref{lindis_2} and \ref{lin_q} for $L^r$-spaces, are given by Figure \ref{fig:M1}.
\begin{figure}[h]
\centering
\begin{tikzpicture}
  \node (thm1p5) {Thm 1.5};
  \node[right=of thm1p5] (thm1p6) {Thm 1.6};
  \node[right=of thm1p6] (thm1p7) {Thm 1.7};
  \node[above=of thm1p6] (thm4p1) {Thm 4.1};
  \node[above=of thm4p1, rectangle, draw] (thm4p3) {Thm 4.3};
  \node[above=of thm1p7] (thm5p2) {Thm 5.2};

  \draw[-implies,double equal sign distance] (thm4p3) to (thm1p5);
  \draw[-implies,double equal sign distance] (thm4p3) to (thm4p1);
  \draw[-implies,double equal sign distance] (thm4p3) to (thm5p2);
  \draw[-implies,double equal sign distance] (thm1p5) to (thm1p6);
  \draw[-implies,double equal sign distance] (thm1p6) to (thm1p7);
  \draw[-implies,double equal sign distance] (thm5p2) to (thm1p7);
\end{tikzpicture}
\caption{General linear operators
}
\label{fig:M1}
\end{figure}
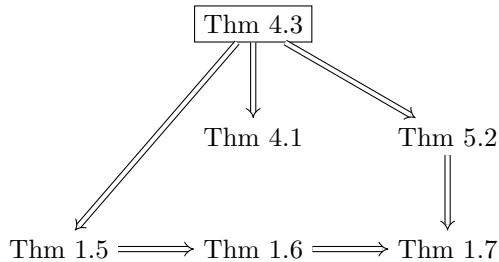
Finally, in an Appendix, we consider the necessity of the conditions we have imposed on the exponents $(p_j)$ in the Basic Question and in Theorems~\ref{answer:qwm_0} and \ref{answer:general_operators}, and we show that they cannot in general be dispensed with. We also show that one cannot avoid the hypothesis of $(p_j)$-convexity in Theorem~\ref{qwm_bb}.

\subsection{Acknowledgements}
T.S.H is supported by the Academy of Finland (through Projects 297929, 314829, and 332740). The authors would like to thank Michael Cowling for bringing reference \cite{FT_P} to their attention, and the referees for their informed and thorough reading of the manuscript and the helpful comments and suggestions which ensued.

\section{Vector-valued disentanglement}\label{sec:vector_valued_disentanglement}
In this section we state and prove two results, both of which are equivalent to the disentanglement result given by Case I of
Theorem~\ref{thmmainbaby}. These will be crucial in the development of both the positive
theory stated in terms of $p$-convexity and of the general linear theory using Rademacher-type. At the end of this section we describe the strategy that we will adopt in order to achieve these aims in the succeeding sections.

\subsection{Functional form}
We first derive an equivalent, arguably more primordial, form of Case I of Theorem~\ref{thmmainbaby}, which makes no reference to saturating positive linear operators, nor to normed lattices, but instead is couched in terms of saturating families of nonnegative measurable functions on a $\sigma$-finite measure space $X$.

\medskip
\noindent
Let $(X, {\rm d} \mu)$ be a $\sigma$-finite measure space. Suppose that for each $1 \leq j \leq d$ we have an indexing set $\mathcal{K}_j$ and a family $\{g_{k_j}\}_{k_j \in \mathcal{K}_j}$ of nonnegative measurable functions on $X$. We assume that, for each $j$, the family  $\{g_{k_j}\}_{k_j \in \mathcal{K}_j}$ {\bf saturates} $X$ 
in the sense that
for every $E \subseteq X$ with $\mu(X) > 0$, there is a subset $E' \subseteq E$ with $\mu (E') > 0$ and a $k_j \in \mathcal{K}_j$ such that $g_{k_j} >0$ on $E'$.

\begin{thm}[Disentanglement of functions]\label{functions}
  With $(X, {\rm d} \mu)$ and $\{g_{k_j}\}_{k_j \in \mathcal{K}_j}$ as above, and $\alpha_j >0$ such that $\sum_{j=1}^d \alpha_j = 1$, assume that
 \begin{equation}\label{prz}
    \int_X \prod_{j=1}^d \left(\sum_{k_j \in \mathcal{K}_j} \beta_{k_j} g_{k_j}\right)^{\alpha_j} {\rm d} \mu
    \leq A \prod_{j=1}^d \left( \sum_{k_j \in \mathcal{K}_j} \beta_{k_j}\right)^{\alpha_j}
 \end{equation}
 for all (finitely-supported) nonnegative $\{\beta_{k_j}\}$. Then there exist nonnegative $\phi_j$ such that
  \begin{equation}\label{prz1}
\prod_{j=1}^d \phi_j(x)^{\alpha_j} \geq 1
  \end{equation}
almost everywhere on $X$, and such that for all $j$,
   \begin{equation}\label{prz2}
    \int_X g_{k_j}(x) \phi_j(x) {\rm d} \mu(x) \leq A
    \end{equation}for all $k_j \in \mathcal{K}_j$.
   \end{thm}

\begin{proof}
Let $\mathcal{Y}_j$ be the normed lattice $l^1(\mathcal{K}_j)$ with counting measure on $\mathcal{K}_j$, whose members  
are denoted by ${\boldsymbol{\beta}}_j =  \{\beta_{k_j}\}_{k_j \in \mathcal{K}_j}$. (There is no requirement on $\mathcal{K}_j$ to be countable.) Define $T_j:  l^1(\mathcal{K}_j) \to \mathcal{M}(X)$ by
  $$ T_j ({\boldsymbol{\beta}}_j) := \sum_{k_j \in \mathcal{K}_j} \beta_{k_j} g_{k_j}.$$
 Note that $T_j$ are saturating positive linear operators. Then \eqref{prz} becomes 
$$ \int_X \prod_{j=1}^d \left(T_j {\boldsymbol{\beta}}_j \right)^{\alpha_j} {\rm d} \mu
  \leq A \prod_{j=1}^d \|{\boldsymbol{\beta}}\|_{\mathcal{Y}_j}^{\alpha_j}.$$
By Case I of Theorem~\ref{thmmainbaby}, there exist $\phi_j$ such that
\eqref{prz1} holds and such that 
$$ \int_X (T_j{\boldsymbol{\beta}}_j)\phi_j {\rm d} \mu \leq A \|{\boldsymbol{\beta}}_j\|_{\mathcal{Y}_j},$$
which is the same as
  \begin{equation*}
    \int_X \left( \sum_{k_j \in \mathcal{K}_j} \beta_{k_j} g_{k_j} \right) \phi_j {\rm d} \mu \leq A  \sum_{k_j \in \mathcal{K}_j} \beta_{k_j},
    \end{equation*}
  or, equivalently, \eqref{prz2}.
    \end{proof}
  
  \medskip
  \noindent
Theorem \ref{functions} can be equivalently rephrased in terms of convex families of functions as follows:
\begin{thm}[Disentanglement of convex families of functions]\label{functions2}
Let $(X, {\rm d}\mu)$ be a $\sigma$-finite measure space. Suppose that $\sum_{j=1}^d \alpha_j = 1$ and that each $\alpha_j > 0$. For each $j \in \{1, \dots , d\}$ let $\mathcal{G}_j$ be a saturating convex set of nonnegative measurable functions. Assume that 
     \begin{equation*}
    \int_X \prod_{j=1}^d g_{j}(x)^{\alpha_j} {\rm d} \mu(x)
    \leq A \quad\text{for all $g_j\in  \mathcal{G}_j$.}
\end{equation*}
 Then there exist nonnegative $\phi_j$ such that
  \begin{equation*}
\prod_{j=1}^d \phi_j(x)^{\alpha_j} \geq 1
  \end{equation*}
almost everywhere on $X$, and such that for all $j$,
   \begin{equation*}
    \int_X g_{j}(x) \phi_j(x) {\rm d} \mu(x) \leq A\quad\text{for all $g_j\in \mathcal{G}_j$.} 
    \end{equation*}
\end{thm}
\begin{proof}The equivalence of Theorem \ref{functions} and Theorem \ref{functions2} is clear from the following observation: writing  $\gamma_{k_j}:= \frac{\beta_{k_j}}{\sum_{k_j\in \mathcal{K}_j} \beta_{k_j}}$ and using homogeneity,
assumption \eqref{prz} of Theorem \ref{functions} can be rephrased as

 \begin{equation*}
    \int_X \prod_{j=1}^d g_{j}^{\alpha_j} {\rm d} \mu
    \leq A \quad\text{for all $g_j\in \conv \mathcal{G}_j$}
 \end{equation*}
 where $\conv \mathcal{G}_j$ is the convex hull of $\mathcal{G}_j$.

\end{proof}

\subsection{Vector-valued form}
The viewpoint of Theorem~\ref{functions} lends itself more readily to applications which are far from obvious from the viewpoint of the formulation of Theorem~\ref{thmmainbaby}. For some of these applications we shall need to work with quasi-normed spaces rather than normed spaces $\mathcal{Y}_j$. We recall that a quasi-normed space $\mathcal{Y}$ is one in which we have the quasi-triangle inequality $\|x + y \|_\mathcal{Y} \leq K (\|x\|_{\mathcal{Y}} + \|y\|_{\mathcal{Y}})$ for some $K \geq 1$ in place of the usual triangle inequality.\footnote{We shall not use the quasi-triangle inequality, and so the constant $K$ will not appear explicitly in our analysis. In fact, every quasi-normed space $\mathcal{Y}$ is $r$-normable and hence has Rademacher-type $r$ for some $0<r\leq 1$; see for example \cite{Kalton_R}. The Rademacher-type constant $R_r(\cy)$ will instead feature. 
}

\medskip
\noindent
For example, we have:

\begin{thm}\label{vector} 
  Suppose that $(X, {\rm d} \mu)$ is a $\sigma$-finite measure space, $\mathcal{Y}_j$ are quasi-normed spaces and $0 <p_j< \infty$. Suppose $T_j: \mathcal{Y}_j \to \mathcal{M}(X)$ are homogeneous of degree $1$ -- that is, $T_j(\lambda f_j) = \lambda T_jf_j$ for all $f_j \in \mathcal{Y}_j$ and all scalars $\lambda$. Assume that for all $j$, the functions $\{|T_j f_j| \, : \, f_j \in \mathcal{Y}_j\}$ saturate $X$. Let $\theta_j >0$ satisfy $\sum_{j=1}^d \theta_j = 1$ and suppose  that we have the $(p_j)$-vector-valued inequality
\begin{equation}\label{vv}
  \int_X \prod_{j=1}^d \left(\sum_{k=1}^N |T_j f_{jk}(x)|^{p_j}\right)^{\theta_j} {\rm d}\mu(x)
  \leq A \prod_{j=1}^d \left(\sum_{k=1}^N \|f_{jk}\|_{\mathcal{Y}_j}^{p_j}\right)^{\theta_j}
  \end{equation}
uniformly in $N$. Then there exist nonnegative $\phi_j$ such that
$$ \prod_{j=1}^d \phi_j(x)^{\theta_j} \geq 1$$
almost everywhere on $X$ and such that for each $j$,
$$ \left( \int_X |T_j f_j(x)|^{p_j} \phi_j(x) {\rm d} \mu(x)\right)^{1/p_j} \leq A^{1/p_j} \|f_j\|_{\mathcal{Y}_j}$$
for all $f_j \in \mathcal{Y}_j$. 
\end{thm}
Notice that we do not need $\mathcal{Y}_j$ to have a lattice structure, nor do we need linearity or positivity of $T_j$. 

\begin{proof}
  Consider the saturating families
  $$\left\{ \left(\frac{|T_j f_j(x)|}{\|f_j\|_{\mathcal{Y}_j}}\right)^{p_j} \, :
  \, f_j \in \mathcal{Y}_j \setminus\{0\} \right\}$$
  of nonnegative functions defined on $X$. Assumption \eqref{vv} translates into \eqref{prz} with
  $\alpha_j = \theta_j$, with the same constant $A$. So by Theorem~\ref{functions} there are nonnegative $\phi_j$ such that
  \eqref{prz1} and \eqref{prz2} hold. And  \eqref{prz2} translates into $ \left( \int_X |T_j f_j(x)|^{p_j} \phi_j(x) {\rm d} \mu(x)\right)^{1/p_j} \leq A^{1/p_j} \|f_j\|_{\mathcal{Y}_j}$
for all $f_j \in \mathcal{Y}_j$.
 \end{proof}

To complete the assertion that Theorem~\ref{thmmainbaby} (Case I), Theorem~\ref{functions} and Theorem~\ref{vector} are all equivalent, we note that Theorem~\ref{vector} implies Case I of Theorem~\ref{thmmainbaby}.
Indeed, the scalar-valued inequality (the hypothesis of Theorem~\ref{thmmainbaby}) readily upgrades to the vector-valued inequality (the hypothesis of Theorem \ref{vector} with $p_j = 1$ for all $j$) via positivity, as follows: we have
\begin{equation*}
\begin{split}
   & \int_X \prod_{j=1}^d \left(\sum_{k} |T_j f_{jk}(x)|\right)^{\theta_j} {\rm d}\mu(x) \leq \int_X \prod_{j=1}^d \left|T_j \left( \sum_k |f_{jk}| \right)(x) \right|^{\theta_j} {\rm d}\mu(x) \\  
&  \leq A \prod_{j=1}^d \norm{\sum_k |f_{jk}|}_{\cy_j}^{\theta_j}\leq A\prod_{j=1}^d \left(\sum_{k} \|f_{jk}\|_{\mathcal{Y}_j}\right)^{\theta_j}.
\end{split}
\end{equation*}
(Note that the use of the triangle inequality for $\mathcal{Y}_j$ here is legitimate since in the implication under consideration the spaces $\mathcal{Y}_j$ are indeed normed spaces.) Summarising, Case I of Theorem~\ref{thmmainbaby}, Theorem~\ref{functions} and Theorem~\ref{vector} are all equivalent.

\medskip
\noindent
The reader will readily verify using H\"older's inequality that the converse statements to Theorem~\ref{functions} and Theorem~\ref{vector} also hold.

\subsection{Vector-valued approach to disentanglement}\label{strat}
We now give a preview of how we shall employ Theorem~\ref{vector} to establish the main disentanglement theorems of the following sections. 
Indeed, thanks to Theorem~\ref{vector} (and its easy converse), given weights $(\theta_j)$ with $\sum_{j=1}^d\theta_j = 1$, exponents $(p_j)$ with $p_j > 0$, a measure space $(X,\mu)$ and linear operators $T_j:\cy_j\to \cm(X)$ defined on quasi-normed spaces $\mathcal{Y}_j$, the following two statements are equivalent:
\medskip
\begin{itemize}
    \item (Disentanglement of $p_j$th powers). The norm inequality
\begin{equation*}
  \int_X \prod_{j=1}^d |T_j f_{j}(x)|^{p_j\theta_j} {\rm d}\mu(x)
  \leq A \prod_{j=1}^d \|f_{j}\|_{\mathcal{Y}_j}^{p_j\theta_j}
  \end{equation*}
  implies that there exist nonnegative $\phi_j$ such that
$ \prod_{j=1}^d \phi_j(x)^{\theta_j} \geq 1$
almost everywhere on $X$ and such that for each $j$,
$$ \left( \int_X |T_j f_j(x)|^{p_j} \phi_j(x) {\rm d} \mu(x)\right)^{1/p_j} \leq \tilde{A}^{1/p_j} \|f_j\|_{\mathcal{Y}_j}.$$

\item (Scalar-valued implies vector-valued inequality). The scalar-valued inequality
\begin{equation*}
  \int_X \prod_{j=1}^d |T_j f_{j}(x)|^{p_j\theta_j} {\rm d}\mu(x)
  \leq A \prod_{j=1}^d \|f_{j}\|_{\mathcal{Y}_j}^{p_j\theta_j}
  \end{equation*}
  implies the vector-valued inequality
\begin{equation*}
  \int_X \prod_{j=1}^d \left(\sum_{k} |T_j f_{jk}(x)|^{p_j}\right)^{\theta_j} {\rm d}\mu(x)
  \leq \tilde{A} \prod_{j=1}^d \left(\sum_{k} \|f_{jk}\|_{\mathcal{Y}_j}^{p_j}\right)^{\theta_j}.
  \end{equation*}
\end{itemize}

 In the following sections, we prove disentanglement theorems via this {\it vector-valued approach}: subject to geometric properties of the spaces $\mathcal{Y}_j$ ($p$-convexity for positive linear operators, Rademacher-type for general linear operators), we deduce the vector-valued inequality from the corresponding scalar-valued inequality, and thereby establish our disentanglement theorems via the equivalence we have just set out.
\section{Positive operators and $p$-convexity}\label{sec:positive_operators_convexity}
In this section we state and prove a more general form of  Theorem~\ref{answer:qwm_0} applying to normed lattices which enjoy $p$-convexity properties.

\begin{definition}[$p$-convexity]
  Let $1 \leq p< \infty$. A normed lattice $\mathcal{Y}$ is $p$-convex if for all finite sequences
  $(f_j)$ in $\mathcal{Y}$ we have
$$ \left\| \left(\sum_j |f_j|^p\right)^{1/p} \right\|_{\mathcal{Y}} \leq C_p(\mathcal{Y}) \left(\sum_j \left\|f_j \right\|_{\mathcal{Y}}^p\right)^{1/p}.$$
  The least such constant is denoted by $C_p(\mathcal{Y})$ and is called the $p$-convexity constant of $\mathcal{Y}$. Clearly $C_p(\mathcal{Y}) \geq 1$.
\end{definition}
Notice that $L^p$ is $p$-convex with $p$-convexity constant equal to $1$, and that every
normed lattice is $1$-convex with $1$-convexity constant equal to $1$. If a lattice $\cy$ is $p$-convex for some $1\leq p<\infty$, then it is $\tilde{p}$-convex  for all $1\leq \tilde{p}\leq p$, see for example \cite{lindenstrauss}.

\medskip
\noindent
Using the fact that $L^r$ is $p$-convex for $1 \leq p \leq r$, with $p$-convexity constant $1$, Theorem~\ref{answer:qwm_0} follows directly from the next, more general result, which is the principal result of this section. This answers our Basic Question for positive linear operators defined on $p$-convex lattices upon taking $\gamma_j = p_j \theta_j$.

\begin{thm}[Disentanglement theorem for positive operators on $p$-convex lattices]\label{qwm_bb}
Suppose that $X$ is a $\sigma$-finite measure space and that $\mathcal{Y}_j$, for $j=1,\dots,d$ are $p_j$-convex normed lattices for some $1 \leq p_j < \infty$. Suppose that the linear operators $T_j: \mathcal{Y}_j \to \mathcal{M}(X)$ are positive,
and that each $T_j$ saturates $X$. Suppose that $\theta_j > 0$ and that $\sum_{j=1}^d \theta_j = 1$.
Finally, suppose that
\begin{equation}\label{abc_1abcd}
\int_X \prod_{j=1}^d (T_jf_j)(x)^{p_j \theta_j }{\rm d} \mu(x)
\leq B \prod_{j=1}^d \Big\|f_j\Big\|_{\mathcal{Y}_j}^{p_j\theta_j}
\end{equation}
for all nonnegative $f_j$ in $\mathcal{Y}_j$, $1 \leq j \leq d$.

\medskip
Then there exist nonnegative measurable functions 
$\phi_j$ on $X$ such that
\begin{equation}\label{factorisebaby_upgrade_aa}
 \prod_{j=1}^d\phi_j(x)^{\theta_j} \geq 1
  \end{equation}
almost everywhere on $X$ and such that for each $j$, 
\begin{equation}\label{controlbaby_upgrade_a}
  \left(\int_X |T_jf_j(x)|^{p_j} \phi_j(x) {\rm d}\mu(x)\right)^{1/p_j}
  \leq B^{1/p_j}C_{p_j}(\mathcal{Y}_j)\|f_j\|_{\mathcal{Y}_j}
\end{equation}
for all $f_j \in \mathcal{Y}_j$.
\end{thm}

\medskip
\noindent
\begin{remark}
The necessity of the geometric assumption that each lattice $\cy_j$ is $p_j$-convex is addressed in 
the Appendix -- see Proposition~\ref{p-convex_nec}.
\end{remark}

\medskip
\noindent
We establish Theorem~\ref{qwm_bb} using the strategy described above in Section~\ref{strat}. Indeed, by the discussion there, and some playing with homogeneities, it suffices to show that under the assumptions of the theorem, the scalar-valued inequality
\begin{equation}\label{sv_1}
  \int_X \prod_{j=1}^d |T_j f_{j}(x)|^{p_j \theta_j} {\rm d}\mu(x)
  \leq B \prod_{j=1}^d \|f_{j}\|_{\mathcal{Y}_j}^{p_j \theta_j}
  \end{equation}
implies the $(p_j)$-vector-valued inequality
\begin{equation}\label{vv_1}
  \int_X \prod_{j=1}^d \left(\sum_{k=1}^N |T_j f_{jk}(x)|^{p_j}\right)^{\theta_j} {\rm d}\mu(x)
  \leq B \prod_{j=1}^d  C_{p_j}(\mathcal{Y}_j)^{p_j \theta_j} \prod_{j=1}^d \left(\sum_{k=1}^N \|f_{jk}\|_{\mathcal{Y}_j}^{p_j}\right)^{\theta_j},
  \end{equation}
and this is exactly what we do in the next lemma:

\begin{lem}[Scalar-valued to vector-valued]\label{scalar_to_vector}
  Suppose that $T_j : \mathcal{Y}_j \to \mathcal{M}(X)$ are positive linear operators and that $\mathcal{Y}_j$ are $p_j$-convex normed lattices for some $p_j \geq 1$. Then \eqref{sv_1} implies \eqref{vv_1}.
    \end{lem}

\medskip
\noindent
Note that when each $\mathcal{Y}_j$ is an $L^{r_j}$-space for $r_j\geq  p_j$, the constant in \eqref{vv_1} is precisely $B$ since then we have $C_{p_j}(L^{r_j})=1$.

\begin{proof}
By homogeneity, we may assume that for each $j$,  $\left(\sum_{k=1}^N \|f_{jk}\|_{\mathcal{Y}_j}^{p_j}\right)^{1/p_j} = 1$.

We are seeking a bound for the left-hand side of 
\eqref{vv_1}, and start by linearising
 the expression $\left(\sum_{k=1}^N |T_j f_{jk}(x)|^{p_j}\right)^{1/p_j}$ in a pointwise manner. We do this by using classical duality for $l^p$ spaces, together with positivity. Indeed, we have, with the $\sup$ taken over all $(\lambda_k)$ with $\sum_k\lambda_k^{p_j'} = 1$,
  $$ \left(\sum_{k=1}^N |T_j f_{jk}(x)|^{p_j}\right)^{1/p_j} = \sup_{(\lambda_k)} |\sum_{k=1}^N \lambda_{k}T_j f_{jk}(x)| = \sup_{(\lambda_k)} |T_j(\sum_{k=1}^N \lambda_{k} f_{jk})(x)| $$
  $$\leq \sup_{(\lambda_k)} T_j \left[\left(\sum_{k=1}^N\lambda_k^{p_j'}\right)^{1/p_j'} \left(\sum_{k=1}^N |f_{jk}|^{p_j}\right)^{1/p_j}\right](x)
= T_j\left[\left(\sum_{k=1}^N |f_{jk}|^{p_j}\right)^{1/p_j}\right](x) := T_jF_j(x).$$
 
  \medskip
  \noindent
Now we are in a position to apply \eqref{sv_1}, and we thus have
  $$ \int_X \prod_{j=1}^d \left(\sum_{k=1}^N |T_j f_{jk}(x)|^{p_j}\right)^{\theta_j} {\rm d}\mu(x)  \leq  \int_X \prod_{j=1}^d T_j F_j(x)^{p_j \theta_j}
{\rm d}\mu(x)\leq B \prod_{j=1}^d \left\| F_{j}\right\|_{\mathcal{Y}_j}^{p_j \theta_j}.$$

\medskip
\noindent
We use the definition of $p$-convexity to obtain
$$  \|F_j\|_{\mathcal{Y}_j} =
\left\|\left[\left(\sum_{k=1}^N |f_{jk}|^{p_j}\right)^{1/p_j}\right]\right\|_{\mathcal{Y}_j}
\leq C_{p_j}(\mathcal{Y}_j) \left(\sum_{k=1}^N \|f_{jk}\|_{\mathcal{Y}_j}^{p_j}\right)^{1/p_j} = C_{p_j}(\mathcal{Y}_j).$$

\medskip
\noindent
Combining these inequalities establishes the lemma.

  \end{proof}
\medskip
\noindent
Notice that we really use linearity of $T_j$ in this argument; sublinearity does not suffice for it to work.

\medskip
\noindent
{\it {Remark.}} The essence of the vector-valued approach to disentanglement lies in upgrading a scalar-valued estimate into the corresponding vector-valued estimate. From the viewpoint of disentanglement of convex families of functions,
this amounts to upgrading the estimate
\begin{equation*}
  \int_X \prod_{j=1}^d |g_{j}(x)|^{\theta_j} {\rm d}\mu(x)
  \leq A \quad\text{for all $g_j\in \mathcal{G}_j$} 
  \end{equation*}
 from the family $$
\mathcal{G}_j:=\mathcal{G}(T_j,\mathcal{Y}_j,p_j):=\left\{ \frac{\abs{T_jf_j}^{p_j}}{\norm{f_j}_{\mathcal{Y}_j}^{p_j}} \right\}
$$ to its convex hull $\conv \mathcal{G}_j$.
Now, Lemma
\ref{scalar_to_vector} loosely states that, under its assumptions, the family $\mathcal{G}_j$ is `essentially convex'. Indeed, let  $\mathcal{F}_1$ and $\mathcal{F}_2$ be sets of non-negative measurable functions and $C>0$ be a constant. Let us write $\mathcal{F}_1\leq C \mathcal{F}_2$ if for each $f_1\in \mathcal{F}_1$ there is $f_2\in \mathcal{F}_2$ such that $f_1\leq C f_2$. Assume that $T:\cy \to \cm(X)$ is a positive linear operator on a $p$-convex normed lattice $\mathcal{Y}$ with $p$-convexity constant  $C_p(\mathcal{Y})$. Then from the definition of $p$-convexity it follows that
$$
\conv \mathcal{G}(T,\mathcal{Y},p)\leq  C_p(\mathcal{Y}) \mathcal{G}(T,\mathcal{Y},p).
$$

\section{General linear operators and Rademacher-type}\label{sec:general_operators_type}
We now consider  general linear (not necessarily positive) operators.
We will follow the same general lines of argument as in the previous section. The key new ingredient
in this setting will be an analogue of the argument of Lemma~\ref{scalar_to_vector} which converts
scalar to vector inequalities, but now without a positivity hypothesis. Once again we shall first need to
linearise the expression $\left(\sum_{k=1}^N |T_j f_{jk}(x)|^{p_j}\right)^{1/p_j}$ in a pointwise
manner. We no longer have positivity at our disposal, so we shall instead use the sequence of Rademacher
functions, which we denote by $(\epsilon_k)$. 

\medskip
\noindent
Let us first suppose for simplicity that each $p_j = 2$.
In this case, we have, for each $j$,
$$\left(\sum_{k=1}^N |T_j f_{jk}(x)|^{2}\right)^{1/2} = \left(\mathbb{E} \left|\sum_{k=1}^N \epsilon_{k} T_j f_{jk}(x)\right|^2\right)^{1/2}$$
$$ \sim_{\theta_j} \left(\mathbb{E} \left|\sum_{k=1}^N \epsilon_{k} T_j f_{jk}(x)\right|^{2\theta_j}\right)^{1/{2\theta_j}} = \left(\mathbb{E} \left|T_j \left(\sum_{k=1}^N \epsilon_{k} f_{jk}\right)(x)\right|^{2\theta_j}\right)^{1/{2\theta_j}}$$
by Khintchine's inequality, so that
$$ \int_X \prod_{j=1}^d \left(\sum_{k=1}^N |T_j f_{jk}(x)|^{2}\right)^{\theta_j} {\rm d}\mu(x) \lesssim_{\{\theta_j\}} \mathbb{E}  \int_X \prod_{j=1}^d \left|T_j \left(\sum_{k=1}^N \epsilon_{jk} f_{jk}\right)(x)\right|^{2\theta_j} {\rm d}\mu(x).$$
If we now assume \eqref{sv_1} with $p_j =2$ for all $j$, we can dominate this last expression by
$$ B ~ \mathbb{E} \prod_{j=1}^d \|\sum_{k=1}^N \epsilon_{jk} f_{jk}\|_{\mathcal{Y}_j}^{2 \theta_j}.$$ 

\medskip
\noindent
If $\mathcal{Y}_j$ is assumed to be of Rademacher-type $2$, that is to say 
$$ \left(\mathbb{E} \left\|\sum_{k=1}^N \epsilon_{k} F_{k}\right\|_{\mathcal{Y}_j}^{2}\right)^{1/2} \leq R_2(\mathcal{Y}_j) \left(\sum_{k=1}^N \| F_{k}\|_{\mathcal{Y}_j}^{2}\right)^{1/2}$$
for some finite $R_2(\mathcal{Y}_j)$, we will obtain (using Jensen's inequality $\mathbb{E}(X^\theta) \leq \mathbb{E}(X)^\theta$ for $0 < \theta < 1$)  
$$ \int_X \prod_{j=1}^d \left(\sum_{k=1}^N |T_j f_{jk}(x)|^{2}\right)^{\theta_j} {\rm d}\mu(x)
\lesssim_{\{\theta_j\}} B~ \prod_{j=1}^d R_2(\mathcal{Y}_j)^{2\theta_j}\prod_{j=1}^d\left(\sum_{k=1}^N
  \|f_{jk}\|_{\mathcal{Y}_j}^2\right)^{\theta_j},$$
which is the analogue of \eqref{vv_1} in this setting. 

\medskip
\noindent
(Note that even in the case that each $\mathcal{Y}_j$ is an $L^2$-space, and so 
$R_2(\mathcal{Y}_j) = 1$, there is an implicit constant greater than one in this last conclusion, due to the use of Khintchine's inequality.)

\medskip
\noindent
The argument now proceeds exactly in accordance with the remarks in Section~\ref{strat}, and we arrive at:

\begin{thm}[Disentanglement theorem for general linear operators on spaces of Rademacher type $2$]\label{qwm_b} 
Suppose that $X$ is a $\sigma$-finite measure space and that $\mathcal{Y}_j$, for $j=1,\dots,d$, are normed spaces which are of Rademacher-type $2$.
Suppose that the linear operators $T_j: \mathcal{Y}_j \to \mathcal{M}(X)$
saturate $X$, and that $\sum_{j=1}^d \theta_j = 1$.
Finally, suppose that
\begin{equation*}
\int \prod_{j=1}^d |T_jf_j(x)|^{2\theta_j }{\rm d} \mu(x)
\leq B \prod_{j=1}^d \Big\|f_j\Big\|_{\mathcal{Y}_j}^{2 \theta_j}
\end{equation*}
for all $f_j$ in $\mathcal{Y}_j$, $1 \leq j \leq d$.
\medskip
Then there exist nonnegative measurable functions 
$\phi_j$ on $X$ such that 
\begin{equation*}
 \prod_{j=1}^d\phi_j(x)^{\theta_j} \geq 1
  \end{equation*}
almost everywhere on $X$, and such that for each $j$, 
\begin{equation*}
  \left(\int_X |T_jf_j(x)|^{2} \phi_j(x) {\rm d}\mu(x)\right)^{1/2}
  \lesssim_{\{\theta_j\}}B^{1/2}R_{2}(\mathcal{Y}_j)\|f_j\|_{\mathcal{Y}_j}
\end{equation*}
for all $f_j \in \mathcal{Y}_j$.
\end{thm}

\medskip
\noindent
The special case of this result when each $\mathcal{Y}_j$ is an $L^{r_j}$-space with $2 \leq r_j < \infty$ is Theorem~\ref{lindis_2}, which immediately follows from Theorem ~\ref{qwm_b} upon using the fact (see below) that the Lebesgue space $L^r$ with $r \geq 2$ has Rademacher-type $2$.

\medskip
\noindent
We now need to discuss what happens when one or more of the $p_j$ are not equal to $2$. We need the notion of Rademacher-type $p$.

\begin{definition}[Rademacher-type]
  Let $0 < p \leq 2$. A quasi-normed space $\mathcal{Y}$ is of Rademacher-type $p$
  if for all finite sequences $(F_k)$ in $\mathcal{Y}$ we have
$$ \left(\mathbb{E} \left\|\sum_{k=1}^N \epsilon_{k} F_{k}\right\|_{\mathcal{Y}}^{p}\right)^{1/p} \leq R_p(\mathcal{Y}) \left(\sum_{k=1}^N \| F_{k}\|_{\mathcal{Y}}^{p}\right)^{1/p}$$
for some finite constant $R_p(\mathcal{Y})$.
\end{definition}

\medskip
\noindent
The least such constant is denoted by $R_p(\mathcal{Y})$ and is called the $p$-Rademacher-type constant of $\mathcal{Y}$. When $ 0 < r \leq 2$, the Lebesgue space $L^r$ has Rademacher-type $p$ for $0 < p \leq r$; when $2<r<\infty$, $L^r$ has Rademacher-type $p$ for $0<  p \leq 2$. Every normed space $\mathcal{Y}$ has Rademacher-type $1$. Note that by Khintchine's inequality, if a quasi-normed space is of Rademacher-type $p$, then it is also of Rademacher-type $\tilde{p}$ for all $0 <\tilde{p} \leq p$. Observe that the one-dimensional normed space $\mathbb{R}$ (and more generally any Hilbert space) has Rademacher-type $2$ with corresponding constant $1$.
When $0<p<1$, Rademacher-type $p$ is equivalent to $p$-normability, i.e. the existence of a constant $C$ such that 
$\| \sum_{k=1}^N F_k \|_{\cy} \leq C (\sum_{k=1}^N \|F_k\|_{\cy}^p)^{1/p}$.

\medskip
\noindent
Ideally we would hope to have:

\begin{asp}
[General disentanglement aspiration for linear operators]\label{qwm_c} 
Suppose that $X$ is a $\sigma$-finite measure space and that  $\mathcal{Y}_j$, for $j=1,\dots,d$, are quasi-normed spaces which are of Rademacher-type $p_j$
for certain $0 <p_j \leq 2 $. Suppose that the linear
operators $T_j: \mathcal{Y}_j \to \mathcal{M}(X)$ saturate $X$, and that $\sum_{j=1}^d \theta_j = 1$.
Finally, suppose that
\begin{equation}\label{def_2}
\int \prod_{j=1}^d |T_jf_j(x)|^{p_j \theta_j }{\rm d} \mu(x)
\leq B \prod_{j=1}^d \Big\|f_j\Big\|_{\mathcal{Y}_j}^{p_j\theta_j}
\end{equation}
for all $f_j$ in $\mathcal{Y}_j$, $1 \leq j \leq d$.
\medskip
Then there exist nonnegative measurable functions 
$\phi_j$ on $X$ such that 
\begin{equation}\label{factorisebaby_upgrade_aaa}
 \prod_{j=1}^d\phi_j(x)^{\theta_j} \geq 1
  \end{equation}
almost everywhere on $X$ and such that for each $j$,
\begin{equation}\label{controlbaby_upgrade_aa}
  \left(\int_X |T_jf_j(x)|^{p_j} \phi_j(x) {\rm d}\mu(x)\right)^{1/p_j}
  \lesssim_{\{\theta_j,  p_j\}} B^{1/p_j}R_{p_j}(\mathcal{Y}_j)\|f_j\|_{\mathcal{Y}_j}
\end{equation}
for all $f_j \in \mathcal{Y}_j$.
\end{asp}

We cannot hope for this to be true in general in situations in which some $p_j<2$, see the Appendix. Nevertheless, we are able to prove
something slightly weaker, namely that the aspiration is in fact a
theorem under the stronger hypothesis that
for those $j$ with $p_j < 2$, the normed spaces $\mathcal{Y}_j$ have Rademacher-type {\em strictly} larger than $p_j$.

\begin{thm}
[Disentanglement theorem for general linear operators on spaces of non-trivial Rademacher type]\label{qwm_ba}
Let $X$ be a $\sigma$-finite measure space and $\mathcal{Y}_j$ quasi-normed spaces. Let $T_j:\mathcal{Y}_j\to \cm(X)$ be linear operators. Suppose that the linear
operators $T_j$ saturate $X$. Let $0 < p_j \leq 2$ and $\sum_{j=1}^d \theta_j = 1$. Assume that
\begin{equation}\label{def_3}
\int \prod_{j=1}^d |T_jf_j(x)|^{p_j \theta_j }{\rm d} \mu(x)
\leq B \prod_{j=1}^d \Big\|f_j\Big\|_{\mathcal{Y}_j}^{p_j\theta_j}
\end{equation}
for all $f_j$ in $\mathcal{Y}_j$, $1 \leq j \leq d$. 

\medskip
Suppose moreover that each space $\mathcal{Y}_j$ has Rademacher-type $r_j=2$
for those $j$ with $p_j =2$, and has Rademacher-type $r_j>p_j$
for those $j$ with $p_j <2$.

\medskip
Then there exist nonnegative measurable functions 
$\phi_j$ on $X$ such that
\begin{equation*}
 \prod_{j=1}^d\phi_j(x)^{\theta_j} \geq 1
  \end{equation*}
almost everywhere on $X$ and such that for each $j$,
\begin{equation*}
  \left(\int_X |T_jf_j(x)|^{p_j} \phi_j(x) {\rm d}\mu(x)\right)^{1/p_j}
  \lesssim_{\{\theta_j, p_j, r_j\}} B^{1/p_j} R_{r_j}(\cy_j)\|f_j\|_{\mathcal{Y}_j}
\end{equation*}
for all $f_j \in \mathcal{Y}_j$.
  \end{thm}
 Using the fact that the Lebesgue space $L^r$ (with $0 <  r < \infty$) has Rademacher-type $\min\{2,r\}$ and hence also Rademacher-type $\tilde{r}$ for every $0 < \tilde{r} \leq \min\{2,r\}$ we immediately obtain Theorem~\ref{answer:general_operators}, (and also the assertion made in the accompanying footnote).

\begin{proof}

  Once again the key issue is to pass from the scalar-valued inequality
  \eqref{def_3} to the vector-valued inequality analogous to \eqref{vv_1}, and this is achieved by linearising the expression
  $$\left(\sum_{k=1}^N |T_j f_{jk}(x)|^{p_j}\right)^{1/p_j}$$
  for each $j$. 
  When $p_j = 2$ the Rademacher functions achieve this, but they are unsuited to do so when $0 < p_j < 2$ and instead we use $p$-stable random variables. (For simplicity of notation, in what follows we shall assume that $p_j < 2$ for all $j$; the easy modifications when $p_j = 2$ for some $j$ are left to the reader.)

  \medskip
  \noindent
 We recall that for $0<p\leq 2$, a real-valued random variable $\gamma$ on a probability space is called (normalised) {\it $p$-stable} if it satisfies $\be(e^{it\gamma})=e^{-\abs{t}^p}$. Note that the distribution (i.e. the pushforward measure on the real line) of a $p$-stable random variable is unique because the characteristic function (i.e. the Fourier transform up to a sign) of a random variable determines its distribution. These random variables enjoy the following key property:
  
  \begin{lemma}[Key property of independent $p$-stable random variables]\label{key}Let $0<q< p\leq 2$. Let $(\gamma_k)$ be a sequence of independent $p$-stable random variables. Then
  $$
\left(  \be \abs{\sum_k \gamma_k a_k}^q\right)^{1/q}\sim_{p,q} \left( \sum_k \abs{a_k}^p\right)^{1/p}
  $$
  for all sequences $(a_k)$ of scalars.
  \end{lemma}

  \medskip
  \noindent
 Pisier proved in \cite{pisier1974} that this property can be upgraded to the vector-valued setting under an appropriate hypothesis of Rademacher-type:
  
  \begin{lemma}[Rademacher-type $r$ implies stable-type $p<r$]\label{lemma:stable_rademacher} Let $0< q<p<r\leq 2$. Let $\mathcal{Y}$ be a quasi-normed space of Rademacher-type $r$. Let $(\gamma_k)$ be a sequence of independent $p$-stable random variables. Then
$$
\left(\be \norm{\sum_{k} \gamma_k f_k }_{\cy}^q\right)^{1/q}\lesssim_{p,q,r} R_r(\cy) \left( \sum_k \norm{f_k}_{\cy}^p \right)^{1/p}
$$
for all sequences $(f_k)$ of vectors.
\end{lemma}
 Note that we need $q<p$ in the above lemmas because $p$-stable random variables fail to be $p$-integrable. 
For a textbook treatment of Rademacher and $p$-stable random variables and Rademacher and $p$-stable types, see for example \cite[Sections 6.2, 6.4, and  7.1]{kalton2016}.
  
  \medskip
  \noindent
  Now, for each $j=1,\ldots,d$ let $(\gamma_{jk})$ be a sequence of independent $p_j$-stable random variables. Then, by Lemma \ref{key},  we have
 $$\left(\sum_{k=1}^N |T_j f_{jk}(x)|^{p_j}\right)^{1/p_j} \sim_{\{\theta_j\}}
  \left(\mathbb{E} \left|\sum_k \gamma_{jk} T_j f_{jk}(x) \right|^{p_j \theta_j}\right)^{1/p_j\theta_j}.$$
  Using this linearisation we can re-phrase the left-hand side of the vector-valued inequality in terms of the left-hand side of the scalar-valued inequality,
  $$ \int_X \prod_{j=1}^d \left(\sum_{k=1}^N |T_j f_{jk}(x)|^{p_j}\right)^{\theta_j} {\rm d} \mu(x)
  \sim_{\{\theta_j\}} \mathbb{E} \int_X \prod_{j=1}^d  \left|\sum_k \gamma_{jk} T_j f_{jk}(x) \right|^{p_j \theta_j}{\rm d} \mu(x)$$
  $$=\mathbb{E} \int_X \prod_{j=1}^d  \left|T_j (\sum_k \gamma_{jk} f_{jk}) (x)\right|^{p_j \theta_j}{\rm d} \mu(x)
  .$$
  Using the assumed scalar-valued inequality \eqref{def_3}, we have the estimate
  $$ \mathbb{E} \int_X \prod_{j=1}^d  \left|T_j (\sum_k \gamma_{jk} f_{jk})(x)\right|^{p_j \theta_j}{\rm d} \mu(x)\leq B \mathbb{E} \prod_{j=1}^d \|\sum_k \gamma_{jk} f_{jk}\|_{\mathcal{Y}_j}^{p_j \theta_j}
  = B \prod_{j=1}^d \mathbb{E} \left(\|\sum_k \gamma_{jk} f_{jk}\|_{\mathcal{Y}_j}^{p_j \theta_j}\right).
$$

By Lemma \ref{lemma:stable_rademacher}, together with the assumption that each space $\cy_j$ has Rademacher-type $r_j>p_j$, and the fact that $\theta_j < 1$, we obtain 
 $$ \mathbb{E}\left(\|\sum_k \gamma_{jk} f_{jk}\|_{\mathcal{Y}_j}^{p_j \theta_j}\right)
  \lesssim_{\theta_j, p_j, r_j} R_{r_j}(\cy_j)^{p_j\theta_j}\left(\sum_k \|f_{jk}\|_{\mathcal{Y}}^{p_j}\right)^{\theta_j}$$
  for each $j$
and therefore
\begin{equation*}
\mathbb{E} \prod_{j=1}^d \|\sum_k \gamma_{jk} f_{jk}\|_{\mathcal{Y}_j}^{p_j \theta_j}
  \lesssim_{\theta_j, p_j, r_j}  \prod_{j=1}^d  R_{r_j}(\cy_j)^{p_j\theta_j}\left(\sum_k \|f_{jk}\|_{\mathcal{Y}}^{p_j}\right)^{\theta_j}.
  \end{equation*}
  
\medskip
\noindent
Summarising, we have proved that if the quasi-normed spaces $\mathcal{Y}_j$ have Rademacher-type $r_j$, then the scalar-valued inequality \eqref{def_3} implies the vector-valued inequality
$$ \int_X \prod_{j=1}^d \left(\sum_{k=1}^N |T_j f_{jk}(x)|^{p_j}\right)^{\theta_j} {\rm d} \mu(x)
\lesssim_{\theta_j, p_j, r_j}  \prod_{j=1}^d  R_{r_j}(\cy_j)^{p_j\theta_j}\left(\sum_k \|f_{jk}\|_{\mathcal{Y}}^{p_j}\right)^{\theta_j}.$$
By the remarks in Section~\ref{strat}, this suffices to establish Theorem~\ref{qwm_ba}.
\end{proof}

\medskip
\noindent
{\em Remark.} Since the linearisation arguments of Theorems ~\ref{qwm_bb} and \ref{qwm_ba} run componentwise, in the case where some of the operators are positive on $p_j$-convex lattices and some non-positive on $r_j$-Rademacher-type normed spaces, we may obtain a hybrid of these two theorems, whose precise formulation we leave to the interested reader.

\section{Multilinear duality and Maurey factorisation extended}\label{revisit}
In this section we apply the two main disentanglement theorems (Theorem~\ref{qwm_bb} for positive linear operators, and Theorem~\ref{qwm_ba} for general linear operators respectively) to deduce multilinear duality and multilinear Maurey factorisation theorems in the spirit of Theorem~\ref{thmmainbaby}.
The treatment we give is very much in parallel to the manner in which Cases II and III of Theorem~\ref{thmmainbaby} can be deduced from Case I.

\medskip
\noindent
Note that Multilinear Maurey factorisation theorems below (Cases III of Theorems~\ref{thm:lq_positive} and \ref{thm:lq_linear}) in the linear case $d=1$  recover  the Maurey factorisation theorems for linear operators \cite{maurey}. We emphasise, however, that our main theorems (Disentanglement Theorems \ref{qwm_bb} and \ref{qwm_ba}) have no linear counterparts since in the case $d=1$ they are vacuous.

\medskip
\noindent
\subsection{Positive operators}
We begin with the setting of positive operators.

\begin{thm}
\label{thm:lq_positive} 
Suppose that $X$ is a $\sigma$-finite measure space and that $\mathcal{Y}_j$ for $j=1,\dots,d$ are $p_j$-convex normed lattices for some $1 \leq p_j < \infty$.
Suppose that the linear
operators $T_j: \mathcal{Y}_j \to \mathcal{M}(X)$ are positive and 
that each $T_j$ saturates $X$. Suppose that $\theta_j > 0$ and that $\sum_{j=1}^d \theta_j = 1$.
Finally, suppose that for some $0 < q \leq \infty$ we have
\begin{equation}\label{abc_1234}
\|\prod_{j=1}^d (T_jf_j)^{p_j \theta_j }\|_{L^q({\rm d} \mu)}
\leq B \prod_{j=1}^d \Big\|f_j\Big\|_{\mathcal{Y}_j}^{p_j\theta_j}
\end{equation}
for all nonnegative $f_j$ in $\mathcal{Y}_j$, $1 \leq j \leq d$.

\medskip
\begin{enumerate}
\item[{\bf Case I.}] {\bf (Disentanglement).} $q = 1$. See Theorem~\ref{qwm_bb}.

\medskip
\item[{\bf Case II.}] {\bf (Multilinear Duality).} If $q >1$, then for every nonnegative $G \in L^{q'}(X)$ there exist nonnegative measurable functions $g_j$ on $X$ such that 
$$ G(x) \leq \prod_{j=1}^d g_j(x)^{\theta_j}$$
almost everywhere, and such that 
$$
 \left(\int_X |T_jf_j(x)|^{p_j} g_j(x) {\rm d}\mu(x)\right)^{1/p_j}
  \leq B^{1/p_j}C_{p_j}(\mathcal{Y}_j)\|G\|_{q'}\|f_j\|_{\mathcal{Y}_j}
$$
for all $f_j \in \mathcal{Y}_j$.

\bigskip
\item[{\bf Case III.}] {\bf (Multilinear Maurey Factorisation).} If $0 < q < 1$ then there exist nonnegative measurable functions $g_j$ on $X$ such that 
$$ \| \prod_{j=1}^d g_j(x)^{\theta_j}\|_{q'} = 1$$
and such that 
$$
 \left(\int_X |T_jf_j(x)|^{p_j} g_j(x) {\rm d}\mu(x)\right)^{1/p_j}
  \leq B^{1/p_j}C_{p_j}(\mathcal{Y}_j)\|f_j\|_{\mathcal{Y}_j}
$$
for all $f_j \in \mathcal{Y}_j$.

\end{enumerate}
\end{thm}

Note that Theorem \ref{thm:lq_positive} in the special case $p_j=1$ for all $j$ is precisely Theorem \ref{thmmainbaby}.

\begin{proof}
We begin with Case II. Suppose that 
\begin{equation*}
\left\|\prod_{j=1}^d (T_j f_j)^{p_j \theta_j}\right\|_{L^q(X)}
\leq B \prod_{j=1}^d \Big\|f_j\Big\|_{\mathcal{Y}_j}^{p_j \theta_j}
\end{equation*}
for all nonnegative $f_j \in \mathcal{Y}_j$, $1 \leq j \leq d$. Then, for all nonnegative
$G \in L^{q'}(X)$ with $\|G\|_{L^{q'}}=1$, we have
\begin{equation*}
\int_X \prod_{j=1}^d (T_jf_j(x))^{p_j \theta_j} G \, {\rm d} \mu(x)
\leq \Big\| \prod_{j=1}^d (T_jf_j)^{p_j \theta_j}\Big\|_q 
\leq B \prod_{j=1}^d \Big\|f_j\Big\|_{\mathcal{Y}_j}^{p_j \theta_j}.
\end{equation*}
It is easy to see that
if $T_j$ saturates $X$ with respect to the measure ${\rm d} \mu$, then it also does so with respect to $G \, {\rm d}\mu$. Moreover, the measure $G \, {\rm d} \mu$ is $\sigma$-finite. 
Therefore, by Theorem~\ref{qwm_bb} 
applied with the measure $G\, {\rm d} \mu$ in place of ${\rm d} \mu$, there are nonnegative measurable
functions $\gamma_j$ such that 
\begin{equation*}
1 \leq \prod_{j=1}^d \gamma_j(x)^{\theta_j} \qquad G \, {\rm d} \mu \mbox{-a.e. on $X$,}
\end{equation*}
and such that for each $j$, 
\begin{equation*}
\left(\int_X |T_jf_j(x)|^{p_j} \gamma_j(x) G(x){\rm d}\mu(x)\right)^{1/p_j} \leq  B^{1/p_j}C_{p_j}(\mathcal{Y}_j) \|f_j\|_{\mathcal{Y}_j}
\end{equation*}
for all $f_j \in \mathcal{Y}_j$. Setting $g_j = \gamma_j G$ gives the
desired conclusion.

\bigskip
\noindent
Now we turn to Case III.
 The main hypothesis \eqref{abc_1234} is that
  $$ \int_X \prod_{j=1}^d (T_jf_j)^{p_j \theta_j q} {\rm d} \mu
  \leq B^q \prod_{j=1}^d \Big\|f_j\Big\|_{\mathcal{Y}_j}^{p_j \theta_j q}$$
  for all nonnegative $f_j \in \mathcal{Y}_j$, $1 \leq j \leq d$.

\medskip
We introduce a new one-dimensional normed lattice $\mathcal{Y}_{d+1}$ with a nonnegative element $y$ of unit norm. Let $T_{d+1}: \mathcal{Y}_{d+1} \to \mathcal{M}(X)$ be given by
$\lambda y \to \lambda {\bf{1}}$ where ${\bf{1}}$ denotes the constant function taking the value $1$ on $X$.

\medskip
Then we have
 $$ \int_X \prod_{j=1}^{d+1} (T_jf_j)^{p_j \theta_j q} {\rm d} \mu
  \leq B^q \prod_{j=1}^{d+1} \Big\|f_j\Big\|_{\mathcal{Y}_j}^{p_j \theta_j q}$$
  for all $f_j \in \mathcal{Y}_j$, $1 \leq j \leq d+1$, where the exponents
  $\theta_{d+1}>0$ and $p_{d+1} > 0$ are at our disposal. 

\medskip
We shall want to impose the condition $\theta_{d+1} = 1/q -1 > 0$ because, with 
$\tilde{\theta}_j := \theta_j q$, we then have $ \sum_{j=1}^{d+1} \tilde{\theta}_j=1$ and
$$\int_X \prod_{j=1}^{d+1} (T_jf_j)^{p_j \tilde{\theta}_j} {\rm d} \mu
  \leq B^q\prod_{j=1}^{d+1} \Big\|f_j\Big\|_{\mathcal{Y}_j}^{p_j \tilde{\theta}_j}$$
  for all $f_j \in \mathcal{Y}_j$, $1 \leq j \leq d+1$.

\medskip
By Theorem~\ref{qwm_bb} we therefore have that there exist $\psi_j$, $1 \leq j \leq d+1$, such that $$ \prod_{j=1}^{d+1} \psi_j(x)^{\tilde{\theta}_j} = 1$$
almost everywhere, and 
$$ \left(\int_X |T_jf_j(x)|^{p_j} \psi_j(x){\rm d}\mu(x)\right)^{1/p_j} \leq B^{q/p_j}C_{p_j}(\mathcal{Y}_j) \|f_j\|_{\mathcal{Y}_j}
$$
for all $f_j \in \mathcal{Y}_j$, $1 \leq j \leq d+1$.

\medskip
The case $j = d+1$ of this last inequality tells us that (if we choose $p_{d+1} = 1$)
$$\int_X \psi_{d+1}(x){\rm d}\mu(x) \leq B^q$$
and, since by the previous equality we have
$$ \psi_{d+1}(x) = \prod_{j=1}^d\psi_{j}(x)^{-\tilde{\theta}_j/\tilde{\theta}_{d+1}} = 
\prod_{j=1}^d\psi_{j}(x)^{-{\theta}_j/{\theta}_{d+1}}
= \prod_{j=1}^d\psi_{j}(x)^{{\theta}_j q'},$$
it gives 
$$ \| \prod_{j=1}^d\psi_{j}(x)^{{\theta}_j}\|_{q'} \geq B^{q/q'}.$$ 
If we now set $g_j = B^{-q/q'}\psi_j$ for $1 \leq j \leq d$  we obtain
$$ \| \prod_{j=1}^d g_{j}(x)^{{\theta}_j}\|_{q'} \geq 1$$ 
and 
$$ \left(\int_X |T_jf_j(x)|^{p_j} g_j(x){\rm d}\mu(x)\right)^{1/p_j} \leq B^{1/p_j} C_{p_j}(\mathcal{Y}_j)\|f_j\|_{\mathcal{Y}_j}$$
for all $1 \leq j \leq d$, and for all $f_j \in \mathcal{Y}_j$.
\end{proof}

\subsection{General linear operators} 
Next we turn to general linear operators, and state a result which in particular contains Theorem~\ref{lin_q}. The proof follows exactly the same arguments as in Theorem~\ref{thm:lq_positive}, with the exception that the application of Theorem~\ref{qwm_bb} there is now replaced by that of Theorem~\ref{qwm_ba}. (We also need for Case III to observe that the one-dimensional normed space $\mathcal{Y}_{d+1}$ which we introduce has Rademacher-type strictly greater than $1$ -- indeed it has Rademacher-type $2$ with constant $1$ as we noted earlier.) We leave the remaining details to the reader. 

\begin{thm}
\label{thm:lq_linear} 
Let $X$ be a $\sigma$-finite measure space and $\mathcal{Y}_j$ quasi-normed spaces. Let $T_j:\mathcal{Y}_j\to \cm(X)$ be linear operators. Suppose that the linear
operators $T_j$ saturate $X$. Let $ 0 <  p_j \leq 2$ and $\sum_{j=1}^d \theta_j = 1$. Assume that for some $0 < q \leq \infty$ we have
\begin{equation*}
\|\prod_{j=1}^d |T_jf_j|^{p_j \theta_j }\|_{L^q({\rm d} \mu)}
\leq B \prod_{j=1}^d \Big\|f_j\Big\|_{\mathcal{Y}_j}^{p_j\theta_j}
\end{equation*}

for all $f_j$ in $\mathcal{Y}_j$, $1 \leq j \leq d$. 

\medskip
Suppose moreover that each space $\mathcal{Y}_j$ has Rademacher-type $r_j =2$
for those $j$ with $p_j =2$, and has Rademacher-type $r_j>p_j$
for those $j$ with $p_j <2$.

\medskip
\begin{enumerate}
\item[{\bf Case I.}] {\bf (Disentanglement).} $q = 1$. See Theorem~\ref{qwm_ba}. 

\medskip
\item[{\bf Case II.}] {\bf (Multilinear Duality).} If $q >1$, then for every nonnegative $G \in L^{q'}(X)$ there exist nonnegative measurable functions $g_j$ on $X$ such that 
$$ G(x) \leq \prod_{j=1}^d g_j(x)^{\theta_j}$$
almost everywhere, and such that 
$$
 \left(\int_X |T_jf_j(x)|^{p_j} g_j(x) {\rm d}\mu(x)\right)^{1/p_j}
  \lesssim_{\{\theta_j, p_j, r_j\}} B^{1/p_j}R_{r_j}(\mathcal{Y}_j)\|G\|_{q'}\|f_j\|_{\mathcal{Y}_j}
$$
for all $f_j \in \mathcal{Y}_j$.

\bigskip
\item[{\bf Case III.}] {\bf (Multilinear Maurey Factorisation).} If $0 < q < 1$ then there exist nonnegative measurable functions $g_j$ on $X$ such that 
$$ \| \prod_{j=1}^d g_j(x)^{\theta_j}\|_{q'} = 1$$
and such that 
$$
 \left(\int_X |T_jf_j(x)|^{p_j} g_j(x) {\rm d}\mu(x)\right)^{1/p_j}
  \lesssim_{\{\theta_j, p_j, r_j\}}B^{1/p_j}R_{r_j}(\mathcal{Y}_j)\|f_j\|_{\mathcal{Y}_j}
$$
for all $f_j \in \mathcal{Y}_j$.

\end{enumerate}
\end{thm}

There are further extensions to Case II in both Theorems~\ref{thm:lq_positive} and \ref{thm:lq_linear} when we replace the role of $L^q$ for $q > 1$ by K\"othe function spaces as in \cite{CHV1}. We leave the details to the interested reader.

\section{Appendix: Why certain conditions are needed}\label{appendix}
At various points in the development of our results we have imposed conditions whose necessity might not be immediately obvious. For example, in the Basic Question we imposed the homogeneity condition \eqref{homog}; in Theorems~\ref{answer:qwm_0} and \ref{answer:general_operators} we imposed upper bounds on the exponents $p_j$; and in Theorem~\ref{qwm_bb} we imposed $p_j$-convexity on the lattices $\mathcal{Y}_j$. In this final section we establish that, in all these cases, the conditions we impose are indeed needed
in order for our results to have a sufficiently broad scope so as to include certain natural examples.

\subsection{Condition \eqref{homog} in the Basic Question}
We first want to clarify to what extent condition \eqref{homog} is needed in the formulation of the Basic Question. 

\begin{prop}\label{homog_nec}
Fix $r_j \geq 1$ and $\gamma_j > 0$ for $1 \leq j \leq d$. Suppose that $(p_j)$
is such that whenever $T_j: L^{r_j}(\mathbb{R}) \to \mathcal{M}(\mathbb{R}^d)$ are positive linear operators such that
\begin{equation}\label{yanxxx}
\int_{\mathbb{R}^d} \prod_{j=1}^d |T_j f_j(x)|^{\gamma_j} {\rm d} x \lesssim \prod_{j=1}^d \|f_j \|_{L^{r_j}(\mathbb{R})}^{\gamma_j}
\end{equation}
holds, then there exists $(\phi_j)$ such that 
\begin{equation}\label{yan111}
\prod_{j=1}^d \phi_j(x)^{\gamma_j/p_j} \geq 1,
\end{equation}
and 
\begin{equation}\label{yan222}
  \left(\int_{\mathbb{R}^d} |T_j f_j(x)|^{p_j} \phi_j(x) {\rm d}x \right)^{1/p_j} \lesssim\|f_j\|_{L^{r_j}(\mathbb{R})}
 \end{equation}
hold. Then $(p_j)$ must necessarily satisfy 
$$ \sum_{j=1}^d \frac{\gamma_j}{p_j} = 1.$$ 
\end{prop}

\begin{proof}
Let $\Phi_j \in L^{\gamma_j}(\mathbb{R}) \setminus \bigcup_{\beta_j \neq \gamma_j}L^{\beta_j}(\mathbb{R})$ and $g_j \in L^{r_j'}(\mathbb{R})$ be nonzero and strictly positive. Let 
$ T_j:L^{r_j}(\mathbb{R}) \to L^{\gamma_j}(\mathbb{R})$ be 
given by 
$$T_j f(s) = \left(\int_\mathbb{R} f g_j\right) \Phi_j(s).$$
Extend $T_j$ to 
$T_j: L^{r_j}(\mathbb{R}) \to \mathcal{M}(\mathbb{R}^d)$
by defining 
$$(T_jf)(x_1, \dots , x_d) := T_jf(x_j).$$ 
Then \eqref{yanxxx} holds with exponents $(\gamma_j)$, but if we replace any $\gamma_j$ by any other exponent, its left-hand side becomes infinite for all nontrivial nonnegative $f_j \in L^{r_j}(\mathbb{R})$. 

\medskip
\noindent
By hypothesis, $(p_j)$ is such that there exists $(\phi_j)$ satisfying
\eqref{yan111} and \eqref{yan222} for this particular $(T_j)$. Let 
$\lambda = \sum_{j=1}^d \gamma_j/p_j$.
Then 
\eqref{yan111} gives 
$$
\prod_{j=1}^d \phi_j(x)^{\gamma_j/\lambda p_j} \geq 1,
$$ 
and so by Lemma~\ref{prelim} we can conclude that 
$$
\int \prod_{j=1}^d |T_j f_j(x)|^{\gamma_j/\lambda} {\rm d} \mu(x) \lesssim \prod_{j=1}^d \|f_j \|_{L^{r_j}}^{\gamma_j/\lambda};$$
that is, \eqref{yanxxx} holds also with exponents $(\gamma_j/\lambda)$ in place of $(\gamma_j)$ for this $(T_j)$. This is a contradiction to what we observed above, unless $\lambda = 1$.  
\end{proof}

\subsection{Sharpness of the exponents in Theorems~\ref{answer:qwm_0} and \ref{answer:general_operators}}
As a preliminary observation, we note that the next two lemmas can be used to demonstrate the sharpness of the exponents arising in the classical  Maurey--Nikishin--Stein theory of factorisation of linear operators.
\begin{lemma}\label{ocho_uno}
For each $1\leq r\leq\infty$ and $0<\gamma<\infty$ we can construct a positive translation-invariant bounded linear operator $T:L^r(\mathbb{G})\to L^\gamma(\mathbb{G})$ (where $\mathbb{G} = \mathbb{T}$ or $\mathbb{R}$ with Haar measure) such that $$\{ 0 < p < \infty : \mbox{ for some nontrivial } \phi, \, T: L^r \to L^p(\phi) \mbox{ boundedly}\} = I_{r, \gamma} := (0, \max\{\gamma, r\}].$$
\end{lemma}

This is well-known. When $\gamma \leq r$ we take $T = I$ and when $\gamma > r$ we take $T$ to be a fractional integral operator (or slight variant thereof when $r=1$). 

\medskip
\noindent
 We next consider general operators.

\begin{lemma}\label{ocho_dos}
For each $1\leq r<\infty$ and $0<\gamma<\infty$ we can construct a translation-invariant bounded linear operator  $T:L^r(\mathbb{G})\to L^\gamma(\mathbb{G})$ (where $\mathbb{G} = \mathbb{T}$ or $\mathbb{R}$ with Haar measure) such that 
$$\{ 0 < p < \infty : \mbox{ for some nontrivial } \phi, \, T: L^r \to L^p(\phi) \mbox{ boundedly}\}$$
$$ = J_{r, \gamma} := 
\begin{cases}
(0, \gamma] &\text{ when } 2\leq \gamma <r \text{ or } \gamma \geq r \\
(0,2] &\text{ when } \gamma<2\leq r \\
(0,r) &\text{ when }  \gamma<r<2.
\end{cases}
$$
\end{lemma}

This is also mostly well-known.  The exponents $\gamma \geq r$ are covered by Lemma~\ref{ocho_uno} (in which case we can take $\mathbb{G} = \mathbb{T}$ or $\mathbb{R}$ with Haar measure), so it remains to consider the exponents $\gamma < r$ (in which case we shall take $\mathbb{G} = \mathbb{T}$). Note that, by an averaging argument, for a translation-invariant operator on a compact abelian group, $T:L^r \to L^p(\phi)$ boundedly for a non-trivial weight $\phi$ if and only if $T:L^r \to L^p(\phi)$ boundedly for the weight $\phi=1$. Thus, 
\begin{equation*}
    \begin{split}
        &\{ 0 < p < \infty : \mbox{ for some nontrivial } \phi, \, T: L^r(\bt) \to L^p(\bt,\phi) \mbox{ boundedly}\}\\
        &=\{ 0 < p < \infty : T: L^r(\bt) \to L^p(\bt) \mbox{ boundedly}\}.
    \end{split}
\end{equation*}

\medskip
\noindent
When $r>2$ we shall also need the following result to assist us in establishing Lemma~\ref{ocho_dos} :
\begin{lem}\label{FTP_0}
Let $2 \leq \gamma < \infty$. Then there is a bounded translation-invariant linear operator $T: L^2(\mathbb{T}) \to L^\gamma(\mathbb{T})$, such that for no $p > \gamma$ is $T$ bounded from $L^\infty(\mathbb{T})$ to $L^p(\mathbb{T})$. 
\end{lem}

For the case $\gamma = 2$ of Lemma~\ref{FTP_0}, an argument based on Rademacher functions can be found in \cite{GCRdeF}, see Chapter VI, Example 2.10(e). The case $\gamma > 2$ follows readily from Bourgain's solution of the $\Lambda(p)$-set problem, \cite{MR989397}. This result states that for each $2 < \gamma < \infty$ there is a set $E \subseteq \mathbb{Z}$ which is a $\Lambda(\gamma)$-set, but which is not a $\Lambda(\tilde{p})$-set for any $\tilde{p} > \gamma$. 
If $T$ is the Fourier multiplier operator with multiplier $\chi_E$, then $T$ is bounded from $L^2(\bt)$ to $L^\gamma(\bt)$ (since $E$ is a $\Lambda(\gamma)$-set) but unbounded from $L^\infty(\bt)$ to $L^p(\bt)$ for every $p>\gamma$ 
(since if $T:L^\infty\to L^p$ boundedly for some $p>\gamma$, then interpolating between this bound 
and the bound $T:L^2\to L^\gamma$ with $\gamma>2$ gives the bound $T:L^q\to L^{\tilde{p}}$ for some 
$q<\tilde{p}$ and $\tilde{p}>\gamma$, which would imply that $E$ is a $\Lambda(\tilde{p})$-set, a contradiction). 
(We thank an anonymous referee for pointing out this connection to us.) Bourgain's argument gives the stronger conclusion that the operator $T$ can also be chosen to satisfy satisfy $T^2 = T$. On the other hand, his argument is not constructive, and so we give a simple constructive proof of Lemma~\ref{FTP_0} -- which is perhaps of independent interest -- in Section~\ref{pf_lem} below.

\medskip
\noindent
We return to the detailed discussion of Lemma~\ref{ocho_dos}.
\begin{itemize}

\item When $2 \leq \gamma <r$ we appeal to Lemma~\ref{FTP_0}, and we take $T$ to {\color{blue}be} a translation-invariant bounded linear operator  $T:L^2\to L^\gamma $ (and hence $T:  L^r  \to L^\gamma$) that is not bounded from $L^\infty$ to $L^p$ for any $p>\gamma$. 

\item When $\gamma < 2 < r$ we appeal to Lemma~\ref{FTP_0}, and we take $T$ to be a translation-invariant bounded linear  operator $T:L^2\to L^2$ (and hence $T: L^r  \to L^\gamma$) that is not bounded from $L^\infty$ to $L^p$ for any $p>2$. 

\item When $\gamma < r$ and $r=2$ we take $T$ to be the identity operator.

\item 
When $\gamma < r < 2$ we appeal to a theorem of Zafran \cite{Zafran} which states that for each $r < 2$ there is a translation-invariant bounded linear operator $T:L^r(\mathbb{T}) \to L^{r,\infty}(\mathbb{T})$ (and thus $T:L^r(\mathbb{T}) \to L^{\gamma}(\mathbb{T})$ for all $\gamma < r$) such that $T$ is not bounded on $L^r$. 
\end{itemize}

\medskip
By taking tensor products we obtain corresponding multilinear examples. Indeed, by choosing operators $T_j:L^{r_j}(\mathbb{G}_j)\to L^{\gamma_j}(\mathbb{G}_j)$ as in Lemmas~\ref{ocho_uno} and \ref{ocho_dos}, and letting the measure space $(X, {\rm d}\mu)$ be the product $X = \mathbb{G}_1 \times \dots \times \mathbb{G}_d$, with $ {\rm d}\mu$ as product measure, we obtain:
\begin{prop}
For each $1\leq r_j\leq\infty$ and $0<\gamma_j<\infty$ there is a $\sigma$-finite measure space $X$ and there are positive linear operators $T_j:L^{r_j}(\mathbb{G}_j) \to\mathcal{M}(X)$ such that
$$ \int_X \prod_{j=1}^d |T_j f_j|^{\gamma_j} \lesssim \prod_{j=1}^d \|f_j\|_{r_j}^{\gamma_j}$$
and such that 
$$\{ (p_j) \in (0,\infty)^d : \mbox{ for each } j, \, T_j: L^{r_j} \to L^{p_j}(\phi_j) \mbox{ boundedly for some nontrivial }\phi_j\}$$ 
$$= \prod_{j=1}^d I_{r_j, \gamma_j}= \prod_{j=1}^d(0, \max\{\gamma_j, r_j\}].$$
\end{prop}

\begin{prop}
For each $1\leq r_j<\infty$ and $0<\gamma_j<\infty$ there is a $\sigma$-finite measure space $X$ and there are linear operators $T_j:L^{r_j}(\mathbb{G}_j) \to\mathcal{M}(X)$ such that
$$ \int_X \prod_{j=1}^d |T_j f_j|^{\gamma_j} \lesssim \prod_{j=1}^d \|f_j\|_{r_j}^{\gamma_j}$$
and such that 
$$\{ (p_j) \in (0,\infty)^d : \mbox{ for each } j, \, T_j: L^{r_j} \to L^{p_j}(\phi_j) \mbox{ boundedly for some nontrivial }\phi_j\}$$ 
$$= \prod_{j=1}^d J_{r_j, \gamma_j}.$$
\end{prop}

\medskip
As immediate corollaries we have:
\begin{cor}\label{seis_tres}
For each $1\leq r_j\leq\infty$ and $0<\gamma_j<\infty$ there is a $\sigma$-finite measure space $X$ and there are positive linear operators $T_j: L^{r_j}(\mathbb{G}_j) \to \mathcal{M}(X)$ such that
$$ \int_X \prod_{j=1}^d |T_j f_j|^{\gamma_j} \lesssim \prod_{j=1}^d \|f_j\|_{r_j}^{\gamma_j}$$
and such that 
$$\{ (p_j) \in (0,\infty)^d : \sum_{j=1}^d \frac{\gamma_j}{p_j} = 1, \mbox{and, for each } j, \, T_j: L^{r_j} \to L^{p_j}(\phi_j) \mbox{ boundedly for some nontrivial }\phi_j\}$$ 
is nonempty if and only if $\sum_{j=1}^d \gamma_j/r_j \leq 1$, and, when this condition holds, equals 
$$
\left(\prod_{j=1}^d(0,r_j]\right) \bigcap \left\{ (p_j) \in (0, \infty)^d : \sum_{j=1}^d \frac{\gamma_j}{p_j} = 1\right\}.$$
\end{cor}

\begin{cor}\label{seis_diez}
For each $1\leq r_j<\infty$ and $0<\gamma_j<\infty$ there is a $\sigma$-finite measure space $X$ and there are linear operators $T_j: L^{r_j}(\mathbb{G}_j) \to \mathcal{M}(X)$ such that
$$ \int_X \prod_{j=1}^d |T_j f_j|^{\gamma_j} \lesssim \prod_{j=1}^d \|f_j\|_{r_j}^{\gamma_j}$$
and such that 
$$\{ (p_j) \in (0,\infty)^d : \sum_{j=1}^d \frac{\gamma_j}{p_j} = 1, \mbox{and, for each } j, \, T_j: L^{r_j} \to L^{p_j}(\phi_j) \mbox{ boundedly for some nontrivial }\phi_j\}$$ 
$$ = \left(\prod_{j=1}^d J_{r_j, \gamma_j}\right) \bigcap \left\{
(p_j) \in (0,\infty)^d : \sum_{j=1}^d \frac{\gamma_j}{p_j} = 1\right\}.
$$
This set is nonempty if and only if we have $\sum_{j=1}^d \gamma_j/\min\{r_j,2\} < 1$ when at least one $r_j < 2$, and $\sum_{j=1}^d \gamma_j \leq 2$ when all $r_j \geq 2$. When nonempty, this set equals 
$$ \left(\prod_{j \, : \, r_j < 2} (0,r_j) \times \prod _{j \,: \, r_j \geq 2} (0,2]\right)\bigcap
\left\{(p_j)\in (0, \infty)^d \, : \, \sum_{j=1}^d \frac{\gamma_j}{p_j}=1\right\}.$$
\end{cor}
These two corollaries establish the assertions concerning sharpness of Theorems~\ref{answer:qwm_0} and \ref{answer:general_operators} which we made in the introduction.

\subsection{Disentanglement implies $p$-convexity}
 
Here we show that the hypotheses of $p$-convexity are intrinsic to Theorem~\ref{qwm_bb}, since $p$-convexity follows from the conclusion of that result, at least in the case when the spaces $\mathcal{Y}_j$ are K\"othe spaces whose duals are norming. This class includes Lorentz spaces and Orlicz spaces.

\medskip
\noindent
We therefore assume in what follows that each $\mathcal{Y}_j$ is a K\"othe function lattice over the $\sigma$-finite measure space $(Y_j, {\rm d}\nu_j)$, and that we can realise the norm of any $f \in \mathcal{Y}_j$ as
$$ \|f\|_{\mathcal{Y}_j} = \sup_{\|g\|_{\mathcal{Y}_j' \leq 1}} |\int_{Y_j} f g \, {\rm d} \nu_j|.$$
We remark that a K\"othe dual $\cy'$ is norming if and only if the pointwise convergence $f_n\uparrow f$ implies the norm convergence $\norm{f_n}_{\cy} \to \norm{f}_{\cy}$ for all pointwise increasing sequences $(f_n)$ (though we shall not need this characterisation here).
\begin{prop}\label{p-convex_nec}
Fix $\cy_j$ as above, and fix $1<p_j<\infty$ for $1 \leq j \leq d$. Assume that there exists a constant $C_{\{\cy_j\}}$ such that for all weights $(\theta_j)$ with $\theta_j > 0$ and $\sum_{j=1}^d \theta_j=1$, all $\sigma$-finite measure spaces $(X,\dmu)$, and all saturating positive linear operators $T_j:\cy_j \to \cm(X)$ the estimate
$$
\int_X \prod_{j=1}^d \abs{T_jf_j(x)}^{p_j \theta_j} \dmu(x) \leq A\prod_{j=1}^d \norm{f_j}_{\cy_j}^{p_j \theta_j} \quad\text{for all $f_j\in \cy_j$}
$$
implies the existence of functions $\phi_j$ such that $\prod_{j=1}^d \phi_j(x)^{\theta_j}\geq 1$ and such that $$ \left(\int_X |T_j f_j|^{p_j} \phi_j \rm d \mu\right)^{1/p_j} \leq C_{\{\cy_j\}} A^{1/p_j} \|f_j\|_{\mathcal{Y}_j}.
$$
Then each space $\cy_j$ is $p_j$-convex.
\end{prop}

\begin{proof} 
Fix $j$. Let $g_j\in \cy_j'$ be of unit norm. Let $(X,\dmu):=(Y_j,\abs{g_j}{\rm d} \nu_j)$. We define $T_j:=I_{\cy_j \to \cy_j}$. 

\medskip
\noindent
For each $i\neq j$, we choose a nonnegative function $F_{i}$ on $Y_i$ such that $\norm{F_{i}}_{\cy_i}=1$. Since $\cy_i'$ is assumed to be norming, for each $\epsilon>0$ we can choose a non-negative function $G_i$ on $Y_i$ with $\|G_i\|_{\cy_i '} =1$ such that $\int_{Y_i} F_i G_i {\rm d} \nu_i \geq (1-\epsilon)\norm{F_{i}}_{\cy_i}=(1-\epsilon)$. We define $T_i: \cy_i \to \mathcal{M}(X)$ by
$$
T_i f (x) = \int_{Y_i} f G_i {\rm d} \nu_i,$$
so that each $T_i f$ is a constant function on $X$. 
Note that $\abs{T_if_i(x)}\leq \norm{f_i}_{\cy_i}$ for all $f_i \in \cy_i$ and that $\abs{T_i F_{i}(x)}\geq (1-\epsilon)$ for all $x\in X$. 

\medskip
\noindent
Let $\theta_j:=\frac{1}{p_j}\in(0,1)$, and choose the remaining $\theta_i\in (0,1)$ in such a way that $\sum_{i=1}^d \theta_i=1$. 

\medskip
\noindent
With these choices, we have
\begin{equation*}
\begin{split}
&\int_X \prod_{i=1}^d \abs{T_if_i(x)}^{p_i \theta_i} \dmu (x)\leq \int_{Y_j} \abs{f_j} \abs{g_j} \dmu_j \prod_{i\neq j} \norm{f_i}^{p_i \theta_i}_{\cy_i}\\
&\leq \norm{g_j}_{\cy_j'} \norm{f_j}_{\cy_j} \prod_{i\neq j} \norm{f_i}^{p_i\theta_i}_{\cy_i} =  \prod_{i=1}^d \norm{f_i}_{\cy_i}^{p_i \theta_i}. 
\end{split}\end{equation*}
By assumption, there are $(\phi_i)$ such that $\prod_{i=1}^d \phi_i(x)^{\theta_i} \geq 1$ and such that for each $i$,
$$ \left(\int_X |T_i f_i|^{p_i} \phi_i \rm d \mu\right)^{1/p_i} \leq C_{\{\cy_j\}}\|f_i\|_{\mathcal{Y}_i}.
$$
Hence, by the equivalence set out in Section~\ref{strat}, we have the  vector-valued inequality
$$
\int_X \prod_{i=1}^d \left( \sum_{k=1}^N \abs{T_i f_{i,k}}^{p_i} \right)^{\theta_i} \dmu \leq C_{\{\cy_j\}} \prod_{i=1}^d \left( \sum_{k=1}^N \norm{f_{i,k}}_{\cy_i}^{p_i} \right)^{\theta_i}
$$
for the same constant $C_{\{\cy_j\}}$. 

\medskip
\noindent
For $i \neq j$, set $f_{i,k} = F_i$ for $k=1$ and $f_{i,k} =0$ for $k=2,\ldots,N$. We obtain,
$$
\int_{Y_j} \left( \sum_{k=1}^N \abs{f_{j,k}}^{p_j} \right)^{1/p_j} \abs{g_j} {\rm d} \nu_j\leq  C_{\{\cy_j\}} \frac{1}{(1-\epsilon)^{d-1}} \left( \sum_{k=1}^N \norm{f_{j,k}}_{\cy_j}^{p_j} \right)^{1/p_j}. $$
By assumption, the K\"othe dual $\cy_j'$ is norming, and hence taking supremum over $g_j$ in the unit ball of $\cy_{j}'$ and letting $\epsilon\to 0$ yields
$$
\norm{ \left( \sum_{k=1}^N \abs{f_{j,k}}^{p_j} \right)^{1/p_j} }_{\cy_j}\leq C_{\{\cy_j\}} \left( \sum_{k=1}^N \norm{f_{j,k}}_{\cy_j}^{p_j} \right)^{1/p_j}.
$$
This is the defining inequality of $p_j$-convexity. The proof is completed.
\end{proof}

\subsection{Constructive proof of Lemma~\ref{FTP_0}}\label{pf_lem}
Finally, we turn to our constructive proof of Lemma~\ref{FTP_0}, which represents a slight strengthening (in the particular case when the underlying group is $\mathbb{T}$) of a result found in 
Fig\`a-Talamanca and Price \cite{FT_P}, Theorem 4.4\footnote{The examples in \cite{FT_P} depend in principle also on the exponent $p$, whereas ours is $p$-independent.}; see also the references therein. 

\medskip
\noindent
We recall (see for example \cite{Katznelson}, p.33) the sequence of Rudin--Shapiro polynomials $P_m$ on $\mathbb{T}$. There is a (deterministic) sequence $a_n \in \{\pm 1\}$ such that the sequence of trigonometric polynomials defined for $m \geq 0$ by 
$$P_m(x) := \sum_{n=0}^{2^m -1} a_n e^{2 \pi i nx}$$
has the following properties (of which the first and the last are trivial and the second is the interesting one):

\begin{itemize}
\medskip
\item $\|P_m\|_2 = 2^{m/2}$

\medskip
\item $\|P_m\|_\infty \leq 2^{(m+1)/2}$

\medskip
\item $2^{(m-1)/2} \leq \|P_m\|_q \leq 2^{(m+1)/2}$ for $1 \leq q \leq \infty$

\medskip
\item $\|\widehat{P_m}\|_\infty =1$.

\end{itemize}
\medskip
\noindent
For the third item, the upper bounds are clear from the second item; for the lower bounds it suffices by H\"older's inequality to show that $\|P_m\|_1 \geq 2^{(m-1)/2}$, and this follows from the first two items together with $\|P_m\|_2 \leq \|P_m\|_1^{1/2} \|P_m\|_\infty^{1/2}$. 

\medskip
\noindent
From the first and fourth of these we deduce by Young's inequality and interpolation that, for $1 \leq r \leq 2$,
$$ \|P_m \ast f\|_2 \leq 2^{m(\frac{1}{r} - \frac{1}{2})} \|f\|_r.$$

\medskip
\noindent
Let $F_m(x) = \sum_{n=0}^{2^m -1} e^{2 \pi i n x}$ so that $\|F_m\|_p \lesssim 2^{m/p'}$ for $1 < p \leq \infty$ and $\|F_m\|_1 \lesssim m$. 

\medskip
\noindent
Observe that 
$P_m \ast F_m = P_m$, so that $\|P_m \ast F_m\|_q = \|P_m\|_q \gtrsim 2^{m/2}$ for all $1 \leq q \leq \infty$. Let $T_m$ denote convolution with $P_m$. Using these bounds we can easily see that 
$\|T_m\|_{L^p \to L^q} \lesssim \|T_m\|_{L^r \to L^2}$ only when $p \geq r$. Indeed, from the upper bounds on 
$\|F_m\|_p$ we deduce that for all $1 \leq p, q \leq \infty$,  $\|T_m\|_{L^p \to L^q}$ 
is bounded below by 
$2^{m(1/2 -1/p')}$ when $p>1$ and $m^{-1}2^{m/2}$ when $p=1$. 

\medskip
\noindent 
We now build an explicit example. We first note that $\tilde{P}_m := e^{2 \pi i 2^{m} x} P_m(x)$
has frequencies in $[2^m, 2^{m+1})$, and similarly with 
$\tilde{F}_m(x) := e^{2 \pi i 2^m x}F_m(x)$.
Performing this modulation does not change any of the estimates on $P_m$ and $F_m$ which we had above, and we have $\tilde{P}_m \ast \tilde{F}_{m} = \tilde{P}_m$ and $\tilde{P}_m \ast \tilde{F}_{m'} = 0$ for $m \neq m'$.

\medskip
\noindent
Fix an $r$ with $1 \leq r \leq 2$. Let $T$ (depending on $r$) be given by convolution with $$ \sum_{m=1}^\infty m^{-2} 2^{m/2} 2^{-m/r} \tilde{P}_m;$$
by the bounds for $P_m$ derived above we see that $T$ is bounded from $L^r$ to $L^2$. 

\medskip
\noindent
Fix $p \geq 1$ and let 
$f_m = m^{-3} 2^{-m/p'} \tilde{F}_{m}$
so that 
$$\|f_m\|_p \leq m^{-3} 2^{-m/p'} \|\tilde{F}_{m}\|_p \lesssim 1$$
uniformly in $m \geq 1$.

\medskip
\noindent
Moreover, we have
$$ Tf_m = m^{-5} 2^{m/2} 2^{-m/r}2^{-m/p'} \tilde{P}_m \ast \tilde{F_m}$$ 
since $\tilde{P}_m \ast \tilde{F}_{m'} = 0$ for $m \neq m'$. 
Therefore,
$$\|T f_m\|_1 = m^{-5} 2^{m/2} 2^{-m/r}2^{-m/p'} \left\|\tilde{P}_m \ast \tilde{F}_{m}\right\|_1\sim 
m^{-5} 2^{-m/r}2^{m/p}$$
for each $m \geq 1$. 

\medskip
\noindent
Consequently,
$$ \|T\|_{L^p \to L^1} \gtrsim \sup_m \|T f_m\|_1 = \infty$$
when $p < r$.

\medskip
\noindent
Thus, for each $ 1 < r \leq 2$, we have built an example of an $L^r \to L^2$-bounded translation-invariant operator $T$ on $\mathbb{T}$, such that for every $1 \leq p < r$, we have $\|T\|_{L^p \to L^1} = \infty$.

\medskip
\noindent
 By duality, for each $ 2 \leq  r < \infty$, we have an explicit example of an $L^2 \to L^r$-bounded translation-invariant operator $T$ on $\mathbb{T}$, such that if $q > r$, we have $\|T\|_{L^\infty \to L^q} = \infty$. This establishes the constructive version of Lemma~\ref{FTP_0}.

\bigskip


\end{document}